\documentclass[11pt, letterpaper, margin=1in]{amsart}

\usepackage{tikz}
\usetikzlibrary{shapes,fit,arrows,positioning,decorations.pathreplacing}

\usepackage[left=1in,right=1in,bottom=0.5in,top=0.8in]{geometry}
\usepackage{amsfonts}
\usepackage{amsmath, amssymb}
\usepackage{graphicx}
\usepackage[font=small,labelfont=bf]{caption}
\usepackage{epstopdf}
\usepackage[pdfpagelabels,hyperindex]{hyperref}
\usepackage{xcolor}
\usepackage{amsthm}
\usepackage{setspace}
\usepackage{float}
\usepackage{pgfplots}
\usepackage{listings}
\usepackage{longtable}
\usepackage{mathrsfs}
\usepackage{cleveref}
\usepackage{booktabs}
\usepackage{makecell}
\usepackage{multirow}
\usepackage{tabularray}
\usepackage{xspace}

\usepackage{thmtools}
\usepackage{thm-restate}

\usepackage[sort]{cite}

\usepackage{todonotes}
\usepackage{mathtools}
\usepackage[english]{babel}

\usepackage{thmtools} 
\usepackage{thm-restate}
\hypersetup{
pdftitle={Tight Bounds for Graph Distances in Scale Free Percolation and Related Models},
pdfsubject={Mathematics, Probability Theory, Combinatorics, Random Graphs, Random Metric Spaces},
pdfauthor={Kostas Lakis, Johannes Lengler, Kalina Petrova, Leon Schiller},
pdfkeywords={}
}

\newtheorem{thm}{Theorem}[section]
\newtheorem*{thm*}{Theorem}
\newtheorem{corollary}[thm]{Corollary}
\newtheorem{prop}[thm]{Proposition}
\newtheorem{lemma}[thm]{Lemma}

\theoremstyle{definition}
\newtheorem{defn}[thm]{Definition}

\theoremstyle{remark}

\definecolor{energy}{RGB}{114,0,172}
\definecolor{freq}{RGB}{45,177,93}
\definecolor{spin}{RGB}{251,0,29}
\definecolor{signal}{RGB}{203,23,206}
\definecolor{circle}{RGB}{217,86,16}
\definecolor{average}{RGB}{203,23,206}

\colorlet{shadecolor}{gray!20}
\pgfplotsset{compat=1.9}

\usepgflibrary{fpu}

\newcommand{\Expected}[1]{\ensuremath{ \mathbb{E} \left[ #1 \right] } }
\renewcommand{\Pr}[1]{\ensuremath{\mathbb{P} \left(#1 \right) }}
\newcommand{\eps}{\varepsilon}

\newcommand{\ZZ}{\mathbb{Z}}

\newcommand{\I}[2]{\mathcal{I}(#1, #2)}

\newcommand{\W}[1]{\ensuremath{ w_{ #1 } }}
\newcommand{\dg}[2]{\ensuremath{ \ensuremath{ d_G(#1, #2) } }}
\newcommand{\pSFP}{\ensuremath{ p_{xy}^{\text{(SFP)}} }}
\newcommand{\pGIRG}{\ensuremath{ p_{xy}^{\text{(GIRG)}} }}

\renewcommand{\d}{\text{d}}
\newcommand{\disjointoccurence}{\ensuremath{\hspace{.1cm} \square \hspace{.1cm} }}

\newcommand{\SFFPP}{FPP\xspace}
\newcommand{\rate}{\lambda}

\newcommand{\costballll}[2]{\mathcal{B}(#1, #2)}

\newcommand{\distsymbol}{r}

\newcommand{\percparam}{\lambda}

\newcommand{\SFP}{SFP}

\newcommand{\alphaparam}{\alpha}
\newcommand{\tauparam}{\tau}

\newcommand{\graphdist}[2]{d_{G}(#1, #2)}
\newcommand{\graphdistcost}[2]{d^{\hspace{.05cm}\text{cost}}_{G}(#1, #2)}

\newcommand{\ceconst}{c_{1}}
\newcommand{\Cconst}{c_2}
\newcommand{\clow}{c_{\text{low}}}
\newcommand{\cupp}{c_{\text{upp}}}
\newcommand{\pconst}{\beta}
\newcommand{\bigDelta}{\Delta}

\newcommand{\myweightof}[1]{w_{#1}}
\newcommand{\geomdist}[2]{|#1 - #2|}

\newcommand{\IH}[3]{ {\myweightof{#1}}^{\alphaparam}{\myweightof{#2}}^{\alphaparam}\Cconst^{-1} \geomdist{#1}{#2}^{-\alphaparam d} (#3 + 1)^{-\pconst}e^{\ceconst {#3}^{1/\bigDelta}}}

\newcommand{\IHabbr}[4]{h(#1, #2, #3, #4)}

\newcommand{\hops}{k}
\newcommand{\funcconst}{C}
\newcommand{\intconst}{C_{\text{int}}}

\newcommand{\KP}[1]{}
\newcommand{\jl}[1]{}
\newcommand{\LS}[1]{}
\newcommand{\KL}[1]{}

\makeatletter
\let\c@equation\c@thm
\makeatother
\numberwithin{equation}{section}

\author[Kostas Lakis]{Kostas Lakis}

\author[Johannes Lengler]{Johannes Lengler}

\author[Kalina Petrova]{Kalina Petrova}

\author[Leon Schiller]{Leon Schiller}

\thanks{Department of Computer Science, Institute of Theoretical Computer Science, ETH Z\"{u}rich, Switzerland.  \texttt{klakis@student.ethz.ch,\{johannes.lengler,kalina.petrova,leon.schiller\}@inf.ethz.ch}. Research of K.P. funded by SNSF grant No.  CRSII5 173721. K.L. gratefully acknowledges support from the John S. Latsis Public Benefit Foundation and the Onassis Foundation. Research of L.S. funded by SNSF grant No. 197138}

\title[Improved Bounds for Polylogarithmic Graph Distances in Scale-Free Percolation]{Improved Bounds for Polylogarithmic Graph Distances in Scale-Free Percolation and Related Models}
\begin{document}

\newcommand{\ultratinypentagram}{
    \begin{tikzpicture}[baseline=-0.8ex]
        \node[regular polygon, regular polygon sides=5, minimum size=2mm, draw, inner sep=0pt, outer sep=0pt, line width=0.05mm] at (0,0) (pentagon) {};
        \foreach \x in {1,2,...,5} {
            \foreach \y in {\x,...,5} {
                \ifnum\x=\y\relax\else
                    \draw[line width=0.05mm] (pentagon.corner \x) -- (pentagon.corner \y);
                \fi
            }
        }
    \end{tikzpicture}
}\textit{}

\newcommand{\tinyKfour}{
    \begin{tikzpicture}[baseline=-0.8ex]
        \node[regular polygon, regular polygon sides=4, minimum size=2mm, draw, inner sep=0pt, outer sep=0pt, line width=0.05mm] at (0,0) (quad) {};
        \foreach \x in {1,2,...,4} {
            \foreach \y in {\x,...,4} {
                \ifnum\x=\y\relax\else
                    \draw[line width=0.05mm] (quad.corner \x) -- (quad.corner \y);
                \fi
            }
        }
    \end{tikzpicture}
}
\newcommand{\CFFPP}{CFFP\xspace}

\begin{abstract}
In this paper, we study graph distances in the geometric random graph models scale-free percolation SFP, geometric inhomogeneous random graphs GIRG, and hyperbolic random graphs HRG. 
Despite the wide success of the models, the parameter regime in which graph distances are polylogarithmic is poorly understood. We provide new and improved lower bounds. In a certain portion of the parameter regime, those match the known upper bounds.

Compared to the best previous lower bounds by Hao and Heydenreich~\cite{hao2023graph}, our result has several advantages: it gives matching bounds for a larger range of parameters, thus settling the question for a larger portion of the parameter space. It strictly improves the lower bounds of~\cite{hao2023graph} for all parameters settings in which those bounds were not tight. It gives tail bounds on the probability of having short paths, which imply shape theorems for the $k$-neighbourhood of a vertex whenever our lower bounds are tight, and tight bounds for the size of this $k$-neighbourhood. And last but not least, our proof is much simpler and not much longer than two pages, and we demonstrate that it generalizes well by showing that the same technique also works for first passage percolation.
\end{abstract}
\maketitle
\thispagestyle{empty}
\clearpage
\setcounter{page}{1}

\newpage

\section{Introduction and Main results}\label{intro}

In the last years, a family of random graph models including \emph{hyperbolic random graphs} (HRG)~\cite{krioukov2010hyperbolic}, \emph{scale-free percolation} (SFP)~\cite{deijfen2013scale}, and \emph{geometric inhomogeneous random graphs} (GIRG)~\cite{bringmann2019geometric}, has emerged as a model for large real-world networks. They combine an underlying geometric space with an inhomogeneous degree distribution. This combination yields many properties that occur in real-world networks across a wide range of domains, such as strong clustering, a rich community structure, ultra-small distance, small separators, compressibility, and more~\cite{bringmann2019geometric}. The networks have been used empirically and theoretically to study algorithms like local routing protocols~\cite{blasius2020hyperbolic,boguna2010sustaining,bringmann2022greedy}, bidirectional search~\cite{blasius2024external} or maximum flow~\cite{blasius2020efficiently}, spreading processes like bootstrap percolation~\cite{koch2021bootstrap} or SI models~\cite{komjathy2020explosion,komjathy2021penalising,komjathy2023four,komjathy2023fourtwo}, and they have been used to study the effectiveness of different interventions during the Covid19 pandemic~\cite{goldberg2021increasing,jorritsma2020not,odor2021switchover}.

Despite this widespread adoption, the fundamental question of graph distances has been open for some parameter regimes. In general, the models come with two parameters: the degrees follow a power-law distribution with exponent $\tau > 1$, i.e., $\Pr{\deg(v) \ge x} \sim x^{1-\tau}$ for any fixed vertex $v$;\footnote{The distribution is often allowed to vary by constant factors or slowly varying functions, but this will not be relevant for this paper.} and $\alpha >1$ determines the number of \emph{weak ties}~\cite{granovetter1973strength} in the network, i.e.\ edges which are present although their geometric distance and degrees suggest otherwise.\footnote{The traditional SFP parameterization uses two parameters $\tau$ and $\gamma$ instead of the one parameter $\tau$. However, one of those parameters is internal to the graph generation process and does not yield additional classes of graphs, which is why we omit the additional parameter.} If $\tau < 3$ then for two random vertices $x,y$ in the giant component, with high probability\footnote{We say an event occurs \emph{with high probability}, or w.h.p., if it occurs with probability $1 - o(1)$.} their graph distance $d_G(x,y)$ is at most doubly logarithmic in their geometric distance, i.e., $d_G(x,y) = O(\log \log |x-y|)$. For $\tau \le 2$ we even have $d_G(x,y) = O(1)$. These regimes are very precisely understood~\cite{abdullah2017typical,bringmann2016average,deijfen2013scale}.\ On the other hand, if $\tau > 3$ and $\alpha >2$, then it is known that the graph distance of two vertices $x,y$ grows linearly with their geometric distance, and this is again well understood~\cite{deprez2015inhomogeneous,Berger_2004}. 

However, in the \emph{polylogarithmic regime} $\tau >3$ and $\alpha\in (1,2)$, the picture is incomplete. It is understood in the limiting case $\tau =\infty$, which is known as \emph{long-range percolation} LRP, that $d_G(x,y) = (\log|x-y|)^{\Delta(\alpha) \pm o(1)}$ where $\Delta(\alpha) = 1/\log_2(2/\alpha) > 1$~\cite{Biskup_2004,Berger_2004}. Since graph distances can only increase with $\tau$, the upper bound applies for any $\tau >1$, and this is the best known upper bound.\footnote{An improved upper bound was claimed in~\cite{hao2023graph}, but the proof turned out to be wrong, see \Cref{sec:wrong_proof} for details.} On the other hand, it is easy to see that for $\tau > 3$ the $k$-neighbourhood can grow at most exponentially, so distances are at least logarithmic, $d_G(x,y) = \Omega(\log|x-y|)$~\cite{deijfen2013scale}. Hence we know that distances are polylogarithmic for $\tau >3$ and $\alpha\in(1,2)$. But the exponents of the upper and lower bound ($\Delta$ and $1$ respectively) did not match, and this gap remained open for a long time. 

Very recently, Hao and Heydenreich~\cite{hao2023graph} could show an improved lower bound of $d_G(x,y) \ge (\log|x-y|)^{\Delta(\min\{\alpha,(\tau-1)/2\}) - o(1)}$, where, as before, $\Delta(x) = 1/\log_2(2/x)$. This closed the gap in the case that $\alpha < (\tau-1)/2$, since then $\min\{\alpha,(\tau-1)/2\} = \alpha$. Their proof had 9 pages and was a complicated application of Biskup's hierarchy argument~\cite{Biskup_2004}. In this paper we give a stronger lower bound with a much simpler inductive proof of only about two pages. It is inspired by ideas of Biskup for the simpler case of LRP \cite{Biskup_2011}, which themselves are adaptations of those in \cite{Trapman_2010}.  More precisely, we show that $d_G(x,y) = \Omega \left((\log|x-y|)^{\Delta(\min\{\alpha,\tau-2 - o(1)\})}\right)$. This closes the gap between upper and lower bound whenever $\alpha < \tau-2$, which comprises a strictly larger portion of the polylogarithmic regime than the bound of~\cite{hao2023graph}. Moreover, throughout the polylogarithmic regime, since $\tau >3$ implies $\tau-2 > (\tau-1)/2$, our bound is strictly stronger in all cases in which the bound of~\cite{hao2023graph} is not tight. 

Our result is also stronger in two other aspects: Firstly, we provide strong tail bounds on the probability $\Pr{d_G(x,y)\le k}$.
These yield a \emph{shape theorem} for the geometric shape of the $k$-neighbourhood of a fixed vertex for growing $k$, in the case $\alpha <\tau-2$ when our lower bound matches the upper bound. The shape theorem also implies that the size of the $k$-neighbourhood of a fixed vertex grows as $e^{k^{1/\Delta \pm o(1)}}$ as $k\to\infty$, which was not known before. Secondly, due to its simplicity, we believe that our method is also potentially easier to generalize. As demonstration, we show that a similar lower bound holds not only for graph distances, but also for first passage percolation. 

In the following, we will start by formally defining the graph models and stating our precise results. The heart of the paper is Section~\ref{sec:SFP}, where we prove our result for the SFP model. In \Cref{sec:BK_inequality_and_proofs}, we explain briefly why the proof carries over with minor modifications for the GIRG model, which in turn includes HRG as a special case. In Section~\ref{sec:FPP}, we extend the proof to first-passage percolation. Finally, in~\Cref{sec:wrong_proof} we explain why the proof of the upper bound claimed in~\cite{hao2023graph} is incorrect.

\subsection{Preliminaries and Random Graph Models}\label{sec:models}
Our results hold for \emph{Scale Free Percolation} (SFP), \emph{Geometric Inhomogeneous Random Graphs} (GIRG), and \emph{Hyperbolic Random Graphs} (HRG). We will define LRP and GIRG formally in this section. HRG has been shown to be a special case of GIRG~\cite{bringmann2019geometric}, so all results proven for GIRG automatically also hold for HRG, and we do not need to formally define HRG. A formal definition of HRG together with its connection to GIRG can be found in~\cite{bringmann2019geometric}. 

\medskip\paragraph{\textbf{Scale Free Percolation (SFP)}} For SFP, we start with the infinite\footnote{Traditionally, LRP is defined on an infinite vertex set while GIRG is defined on a finite vertex set, following the tradition of mathematics for LRP and of computer science for GIRG. However, both models can be defined in either a finite or an infinite version, and this does not affect graph distances, see~\cite{komjathy2020explosion} for details.} $d$-dimensional grid, which is our set of vertices. For two points $x, y \in \mathbb{Z}^d$, we define their distance $|x - y|$ via the usual Euclidean norm. Moreover, each vertex draws a \emph{weight} $\W{x}$ independently identically distributed from a power-law distribution. For our purposes, this is a Pareto distribution satisfying
\begin{align*}
    \Pr{\W{x} \ge z} = z^{1-\tau}
\end{align*} for $z \ge 1$, where the parameter $\tau > 1$ is the \emph{power-law exponent}.

We add edges in two different ways. First, we place the usual \emph{grid edges}\footnote{Some variants of SFP do not include grid edges. Since we prove lower bounds on graph distances, we make our result stronger by including grid edges.} between points that are adjacent in $\mathbb{Z}^d$, i.e., we place an edge between $x = (x_1, \ldots, x_d)$ and $y = (y_1, \ldots, y_d)$ if there is some coordinate $1 \le i \le d$ such that $|x_i - y_i| = 1$ and $x_j = y_j$ for all $j \neq i$. Second, we randomly create long-range edges, also called \emph{weak ties}, by placing an edge between $x, y \in \mathbb{Z}^d$ with probability $p_{xy}$, which is defined as\footnote{In the literature, the connection probability is often defined as $1 - \exp( - \lambda |x- y|^{-\alpha d} )$. We remark that this differs from our connection probability at most by a constant factor and that our proof techniques are sufficiently robust to also deal with this setting. } 
\begin{align}\label{eq:SFP}
    \pSFP \coloneqq \min\left\{ 1 , \lambda \left( \frac{w_xw_y}{|x - y|^{ d}} \right)^\alpha \right\},
\end{align}
where $\lambda > 0, \alpha > 1$ are constants. We write $x \sim y$ if there is an edge between $x, y \in \ZZ^d$ and $x \nsim y$ otherwise. 

We remark that the parameterization is slightly different from the original one used in~\cite{deijfen2013scale} and also in~\cite{hao2023graph}. Compared to our formulation, they used auxiliary weights $w_x' := w_x^\alpha$, which were then drawn from a power-law distribution with exponent $\tau' = 1+d(\tau-1)/\alpha$. These two parameterizations are equivalent, but our formulation has the advantage that the weights correspond, up to constant factors, to the expected degree of the vertices, $\Expected{\deg(x) \mid w_x} = \Theta(w_x)$. Our formulation also saves an internal parameter of the model. Moreover, we rescaled $\alpha$ by a factor $d$ to match it with the parameterization of GIRG. When comparing our results with the lower bounds in~\cite{hao2023graph}, the following transformations are needed, where the subscript ``\hspace{1sp}\cite{hao2023graph}'' indicates notation from that paper, and parameters without subscript are from our paper.
\begin{align}\label{eq:Hao_parameters}
    \alpha_{\text{\hspace{1sp}\cite{hao2023graph}}} = \alpha d \qquad\text{ and }\qquad \gamma_{\text{\hspace{1sp}\cite{hao2023graph}}} = \tau-1.
\end{align}
The internal parameter $\tau_{\text{\hspace{1sp}\cite{hao2023graph}}}$ is superfluous and does not have a correspondence in our paper.

\medskip

\paragraph{\textbf{Geometric Inhomogeneous Random Graphs (GIRG)}}

The main difference between SFP and GIRG is that in GIRG, the positions of vertices are randomly chosen. Namely, for some large enough $n \in \mathbb{N}$, we consider a cube $\mathcal{X}$ of volume $n$ in $\mathbb{R}^d$, where $d$ is a constant. Our vertex set $V$ then consists of $n$ vertices, where the position $\xi_x$ of each vertex $x$ is picked independently at random from the uniform distribution over $\mathcal{X}$. The distance between two vertices $x,y \in V$ is defined as the Euclidean norm  $|\xi_x - \xi_y|$, as in SFP. Again similarly to SFP, each vertex $x$ draws a weight $w_x$ independently from the Pareto distribution satisfying 
\begin{align*}
    \Pr{\W{x} \ge z} = z^{1-\tau}
\end{align*} for $z \ge 1$, with $\tau > 1$. Finally, two different vertices $x$ and $y$ are connected by an edge with probability\footnote{The original paper~\cite{bringmann2019geometric} used a geometry that was rescaled by a factor $n^{1/d}$. I.e., they used a cube of volume one and had an additional factor $n$ in the denominator of~\eqref{eq:GIRG}. Both variants are equivalent, but our scaling aligns better with the SFP scaling and allows GIRGs to be extended to infinite graphs if desired~\cite{komjathy2020explosion}.}
\begin{align}\label{eq:GIRG}
\pGIRG = \Theta\left(\min\left\{ 1 , \left( \frac{w_xw_y}{|\xi_x - \xi_y|^{ d}} \right)^\alpha \right\}\right),
\end{align}
where the hidden constants are uniform over all $x,y$. Formally, we require that there are two absolute constants $\clow,\cupp >0$ independent of $n$ such that for all $n$ and any two different vertices $x,y\in V$, conditional on their weights $w_x,w_y \ge 1$ and positions $\xi_x,\xi_y\in \mathcal X$, 
\begin{align*}
\clow \min\left\{ 1 , \left( \frac{w_xw_y}{|\xi_x - \xi_y|^{ d}} \right)^\alpha \right\} \le \pGIRG \le \cupp\min\left\{ 1 , \left( \frac{w_xw_y}{|\xi_x - \xi_y|^{ d}} \right)^\alpha \right\}.
\end{align*}
The reason for allowing constant factor deviations is that then hyperbolic random graphs (HRG) is a special case of GIRG with $d=1$, where the Euclidean distance is replaced by the angular distance in hyperbolic space~\cite{bringmann2019geometric}. Hence, any statement proven for this version of GIRG also holds for HRG.

Note that the constants $\tau$ and $\alpha$ have an analogous role here as in SFP. A notable difference to SFP is that in GIRG there are no grid edges (since positions are no longer on the grid), but that does not significantly affect our results. 
\medskip

\paragraph{\textbf{Long Range Percolation (LRP)}}

To put our results in an adequate context, we sometimes reference the model of \emph{Long Range Percolation} (LRP) which is a special case and predecessor of SFP, cf. e.g. \cite{Biskup_2004, Biskup_Lin_2019}. We define LRP completely analogously to SFP with the only modification that every vertex has a (deterministic) weight of $1$.

\medskip

\paragraph{\textbf{First Passage Percolation (FPP)}}

Our lower bounds on graph distances are (in a similar form) also applicable to \emph{First Passage Percolation} (FPP) on the graph models SFP, GIRG, HRG, and also LRP. 
Here, each edge $e$ of the graph draws a random length or cost $c_{e}$ independently identically distributed (i.i.d.) from a distribution over the non-negative reals. For two vertices $x,y$ we are then interested in the \emph{minimal/infimal cost} of all paths from $x$ to $y$. This can be done on an arbitrary finite or infinite underlying graph. In this work we restrict our attention to the case where each edge cost is sampled from an exponential distribution with rate $1$. The formal definition is as follows.

\begin{defn}[First Passage Percolation (FPP)]\label{def:fpp}
    We call the following process \emph{First Passage Percolation}, in short FPP. Given a graph $G=(V,E)$, assign to each edge $e$ an i.i.d.\ \emph{cost} $c_e$ sampled from an exponential distribution with rate 1. For a finite path $\pi$, we define the \emph{cost} of that path as
    \[c(\pi) = \sum\limits_{e \in \pi}c_e\]
    Then, the \emph{cost-distance} or \emph{first passage time} between two vertices $x$ and $y$ is the minimum (or infimum) cost of any finite path connecting $x$ and $y$, i.e. 
    \[\graphdistcost{x}{y} := \inf\{c(\pi) : \pi \in \mathcal{P}_{x,y}\} \quad \text{for} \quad x,y \in V,\]   
    where $\mathcal{P}_{x,y}$ is the set of all finite paths between $x$ and $y$.
\end{defn}

Recently, FPP on SFP, GIRG, and HRG was studied \cite{komjathy2020explosion, komjathy2023four, komjathy2023fourtwo, completeFPP}. 
In particular, it was shown in \cite{komjathy2020explosion, komjathy2023four} that \SFFPP on those graphs exhibits an \emph{explosive} behaviour if the vertex weights have infinite variance (i.e.\ if $\tau <3$). This means that the cost-distance between two vertices $x,y$ converges in distribution against a random variable that is finite almost surely. In the infinite model SFP, this means that there are infinitely many vertices reachable within finite cost from a given vertex; in the finite models GIRG and HRG, a constant fraction of all vertices have cost-distance $\mathcal{O}(1)$. We emphasize that this is not true for \emph{graph distances} for $\tau\in (2,3)$ since then degrees are finite almost surely, so the number of vertices in graph distance $C$ is finite/constant for any constant $C>0$.

We remark that a model with a similar name, \emph{long-range first passage percolation} was introduced and studied in~\cite{completeFPP}. This is \emph{not} FPP on  LRP, but a different model with the complete graph on $\ZZ^d$ (all edges are present, degrees are infinite), where transmission times are penalized for edges between vertices of large Euclidean distance. Despite the differences, the models are related, and our analysis of FPP uses a coupling to a similar model and is inspired by~\cite{completeFPP}.

\paragraph{\textbf{Terminology}}

We are interested in the typical \emph{graph distance} $d_G(x,y)$ between two vertices $x,y$ in SFP, GIRG, and HRG, and the \emph{cost-distance} $\graphdistcost{x}{y}$ in FPP.\footnote{Formally, the graph distance is also defined in \Cref{def:fpp} by setting $c(e):= 1$ for all edges $e$.}
We are interested in the asymptotic behavior of $\dg{x}{y}$ and $\graphdistcost{x}{y}$ in terms of the Euclidean distance $|x- y|$. That is, we study how $\dg{x}{y}$ and $\graphdistcost{x}{y}$ scale as functions of $|x - y|$ when $|x - y| \rightarrow \infty$. We define the function $\Delta(\beta) := \frac{1}{\log_2(2/\beta)}$, which will appear in the exponent governing the polylogarithmic behavior of graph distances. Throughout, we use $\Delta = \Delta(\alpha)$ where $\alpha$ is the long-range parameter of the relevant model.

Previous work showed that both LRP and SFP exhibit multiple phase transitions in the asymptotic behavior of $\dg{x}{y}$ depending on the model parameters, which were briefly summarized in the introduction and are further summarized (together with our results) in \Cref{tab:results}.

\begin{table}[H]
   \small
   \centering
   
   {\renewcommand{\arraystretch}{1.3}
   \setlength{\extrarowheight}{2pt}
   \begin{tabular}{l||c|c|c|c}
   \toprule\toprule
   \textbf{Model} & \multicolumn{2}{c}{\textbf{$\alpha \in (1,2)$}} & \multicolumn{2}{|c}{ $\alpha > 2$ } \\ 
   \hline
   \hline
   LRP & \multicolumn{2}{c}{
   \makecell{
        $\Theta\left(\log(|x-y|)^{\Delta(\alpha)}\right)$ \cite{Biskup_2004, Biskup_Lin_2019}
   }
   } & \multicolumn{2}{|c}{
        \makecell{
            $\Theta(|x-y|)$ \cite{Berger_2004}
        }
    }\\
   \hline
   \hline
     & $\tau \in (2,3)$ & $\tau > 3$ & $\tau \in (2,3)$ & $\tau > 3$\\ \hline
    SFP & \makecell{ 
        $\Theta(\log\log(|x-y|))$  \\ \cite{deijfen2013scale}
    } & \makecell{
        $\le \log(|x-y|)^{\Delta(\alpha) + o(1)}$ \cite{Biskup_2004, Biskup_Lin_2019} \footnote{}\\
        \boldmath
        $\ge \log(|x-y|)^{\Delta(\min\{ \alpha, \tau - 2\}) - o(1) }$ \\ \textbf{\Cref{cor:lowerboundSFP}}
        \unboldmath
    }& \makecell{$\Theta(\log\log(|x-y|))$ \\ \cite{deijfen2013scale}} & \makecell{$\Theta(|x - y|)$ \\\cite{Berger_2004}} \\
   \hline
   FPP on SFP & $\Theta(1)$ \cite{komjathy2023four} & \makecell{ $ \le \log(|x-y|)^{\Delta(\alpha) + o(1)}$ \cite{komjathy2023four} \\
   \boldmath
   $\ge \log(|x-y|)^{\Delta(\min\{\alpha, \frac{\tau - 1}{2}\}) - o(1)}$ \\ \textbf{\Cref{cor:lowerboundFPP}}
   \unboldmath
   } & $\Theta(1)$ \cite{komjathy2023four} & ? \\
   \bottomrule
   \bottomrule
   \end{tabular}}

   \begin{align*}
   \end{align*}
   \caption{Upper and lower bounds for graph distances and cost-distances in long-range percolation (LRP),
   scale-free percolation (SFP) and first passage percolation (FPP) on SFP, with $\Delta(\beta) = 1/\log_2(2/\beta)$. The results on graph distance for SFP also hold for geometric inhomogeneous random graphs (GIRG) and hyperbolic random graphs (HRG). Our results are indicated in bold.}   \label{tab:results} 
\end{table}\footnotetext{In \cite{hao2023graph}, the authors claim to prove an improved logarithmic upper bound with exponent $\Delta(\min\{ \alpha, \tau - 2 \})$, which would match our lower bounds. However, we show that their proof is wrong, see \Cref{sec:wrong_proof}.}

\subsection{Our results}
In this section, we formally state our main results. To keep the exposition simple, we state them only for SFP, not for GIRG and HRG. All results on graph distances in this section also hold for GIRG and HRG, where all constants can be chosen independently of the number $n$ of vertices. We give details on that in~\Cref{sec:proof-GIRG}.

We prove stronger lower bounds for graph distances in the logarithmic regimes of SFP with a much simpler proof than in \cite{hao2023graph}. Furthermore, we also obtain a \emph{shape theorem} which sandwiches the $k$-neighbourhood of a given vertex between two geometric balls of similar size. The key result is a general upper bound on the probability that two vertices have graph/cost-distance at most $k$. For the following claims, recall that $\Delta(\beta) = 1/\log_2(2/\beta)$.

\begin{restatable}[Tail Bound for Graph Distances in SFP]{thm}{mainSFPlowerboundTheorem}\label{thm:mainSFPlowerboundTheorem}
Consider SFP with parameters $\alphaparam{} \in (1, 2), \tauparam{} > 3$ and $\percparam$. 
Fix any sufficiently small $\eps > 0$ and let $\Delta' = \Delta(\min\{\alphaparam, \tauparam - 2 - \eps\})$. Then, there exist constants $\ceconst{}, \Cconst{}, \pconst{}$ depending on the model parameters as well as $\eps$ such that for any pair of vertices $x, y \in \mathbb{Z}^d$ and any $k \in \mathbb{N}$, we have
\begin{align*}
    \Pr{\graphdist{x}{y} \le \hops} \le \Cconst^{-1} \geomdist{x}{y}^{-\alphaparam d} (\hops + 1)^{-\pconst} \exp\left(\ceconst {\hops}^{1/\Delta'}\right).
\end{align*}
\end{restatable}

The above bound is proved by induction over $k$ in \Cref{sec:SFP}. We express the probability that a path of length $\le k$ exists recursively by decomposing the path into two parts connected by an edge which has the longest geometric distance on the original path. Applying the inductive hypothesis and integrating over all possible endpoints and weights of the endpoints of said edge then yields the desired bound. However, this actually only works if $\alpha < \tau - 2$, which is needed to ensure convergence of some involved integrals. We remedy this and make the theorem applicable also to the case $\alpha \ge \tau - 2$ by using a coupling argument. Here (see Lemma~\ref{lem:changealpha} for the exact statement), we argue that decreasing $\alpha$ only makes the model denser, so graph distances can only become smaller. Thus, to prove the claim for some $\alpha\ge \tau - 2$, we can decrease $\alpha$ to some value $\alpha' := \tau-2-\eps$ without increasing distances, and then apply \Cref{thm:mainSFPlowerboundTheorem} for the already settled case $\alpha'$. This is the reason why we use the exponent $\Delta(\min\{\alpha, \tau - 2 - \varepsilon\})$. Using the above tail bound then directly implies a bound on typical graph distances.

\begin{restatable}[Typical Graph Distances in SFP]{corollary}{lowerbounddistancesinSFP}\label{cor:lowerboundSFP}
     Consider SFP under the assumption $\alpha \in (1,2)$ and $\tau > 3$. Then for every sufficiently small $\varepsilon > 0$, there is a constant $c > 0$ such that \begin{align*}
        \lim_{|x-y| \rightarrow \infty} \Pr{ \dg{x}{y} \ge c \log(|x-y|)^{ \Delta(\min\{\alpha, \tau - 2-\varepsilon \})} } = 1.
     \end{align*}
\end{restatable}

Note that in the case $\alpha < \tau - 2$, the exponent is exactly $\Delta(\alpha)$ (if $\varepsilon$ is chosen sufficiently small), which is a slightly stronger result than the one we obtain if $\alpha \ge \tau - 2$ and matches the known upper bounds in \cite{Biskup_Lin_2019} up to only a constant factor in front of the $\log$. The previously best lower bounds by Hao and Heydenreich \cite{hao2023graph} only matched these upper bounds if $2\alpha < \tau - 2$ and nonetheless were only tight if we ignore an additional additive constant of $-\varepsilon$ in the exponent. If $\alpha \ge \tau - 2$, we also have to account for such an $\varepsilon$ in the exponent, but even in this case,
our result strengthens the lower bounds in \cite{hao2023graph} and at the same time relies on a significantly simpler proof. The original proof heavily relied on so-called \emph{hierarchies} as introduced in \cite{Biskup_2004} and required complex combinatorial estimates for showing that certain structures w.h.p.\ do not exist. We avoid this by using the inductive proof strategy as described above instead. This not only simplifies and improves the existing lower bounds on graph distances, but the tail bound in \Cref{thm:mainSFPlowerboundTheorem} further yields a so-called \emph{shape theorem} precisely characterizing the diameter and the cardinality of the \emph{$k$-neighbourhood} of a vertex $x$, i.e., the set of vertices in graph distance at most $k$ from $x$. To this end, we define $\mathcal{B}(x, k) \coloneqq \{ y \in V \mid \dg{x}{y} \le k \}$ as the set of all $k$-hop neighbors of a given vertex $x$. 

\newcommand{\nearvertices}{\mathcal{X}_{low}}
\newcommand{\farvertices}{\mathcal{X}_{upp}}

\newcommand{\shapeEqTitle}{A}
\newcommand{\sizeEqTitle}{B}
\begin{restatable}[Shape Theorem for $k$-Balls in SFP]{thm}{fullshapetheorem}\label{thm:fullshapetheorem}
    Consider SFP with $\alphaparam < \tauparam - 2$. Let $\Delta = \Delta(\alpha)$, fix an $\eps > 0$ and let $\nearvertices{}, \farvertices{}$ be the set of vertices at a geometric distance of at most $q(k) = e^{k^{1/\Delta - \eps}}$ and at most $r(k) = e^{k^{1/\Delta + \eps}}$ from a fixed vertex $x$, respectively. Then, we have
    \begin{align*}
      \tag{\shapeEqTitle{}}
      \lim_{k \to \infty} \Pr{\nearvertices{} \subseteq \mathcal{B}(x, k) \subseteq \farvertices{}} = 1.
    \end{align*}
    In particular, 
    \begin{align*}
      \tag{\sizeEqTitle{}}
      \lim_{k\to\infty} \Pr{e^{k^{1/\Delta - \eps}} \le |\mathcal{B}(x, k)| \le e^{k^{1/\Delta + \eps}}} = 1.
    \end{align*}
\end{restatable}

The lower bound of the theorem comes from~\cite{Biskup_2011}, while we contribute the upper bound. Note that lower bounds on graph distances correspond to upper bounds for $\mathcal{B}(x, k)$ and vice versa. 

We further note that while our upper bound in Theorem~\ref{thm:fullshapetheorem} also holds for GIRG and HRG, the lower bound is not known for these models and can only hold with some caveats. Firstly, due to the lacking grid edges those graphs are not connected, and require additional constraints to ensure the existence of a giant (linear-size) connected component. Second, even if the giant component exists, a constant fraction of vertices are not in the largest component. Hence, the lower bounds in (\shapeEqTitle{}) and (\sizeEqTitle{}) can only hold conditioned on $x$ being in the largest component, and for (\shapeEqTitle{}) we must intersect $\nearvertices$ with the giant component. We conjecture that the lower bounds in Theorem~\ref{thm:fullshapetheorem} hold with these caveats, but this is not known.

\subsection{First passage percolation on SFP}

Using similar techniques and inspired from those in \cite{completeFPP}, we can prove similar statements for FPP on SFP. Analogous to \Cref{thm:mainSFPlowerboundTheorem}, we obtain a tail bound on the probability that $x,y$ have cost-distance at most $t$. We remark that the same result could be obtained analogously for GIRG (and thus HRG) as well, but we omit this for conciseness. 

\begin{restatable}[Tail Bound for Cost-Distances for FPP on SFP]{thm}{lowerboundSFPsquared}\label{thm:lowerboundSFPsquared}
    Consider \SFFPP{} on SFP and arbitrary vertices $x, y$. Fix any sufficiently small $\eps > 0$. There exists a constant $c$ depending only on $\alphaparam$, $\eps$ and $\tauparam$ such that for $\Delta'' = \Delta(\min\{\alphaparam, \frac{\tauparam - 1}{2}\} - \eps)$,
    \begin{align*}
        \Pr{\graphdistcost{x}{y} \le t} \le |x - y|^{-\alphaparam d} \exp\left(ct^{1/\Delta''}\right).
    \end{align*}
\end{restatable}

The proof of this differs from \Cref{thm:mainSFPlowerboundTheorem} in some significant aspects which are formally presented in \Cref{sec:FPP}. Intuitively, the main differences are as follows. Firstly, we now have two sources of randomness: the existence of edges and the cost of an existing edge. The first step towards proving \Cref{thm:lowerboundSFPsquared} is therefore to combine these two sources into a single one. This is achieved by coupling the model to a related model called \emph{Complete Scale Free First Passage Percolation} or \CFFPP{} for short. Here, all possible edges on the vertex set $\ZZ^d$ exist a priori but the cost of the edge between $x,y \in \ZZ^d$ is now drawn from an exponential distribution with rate $w_x^\alpha w_y^\alpha |x - y|^{-\alpha d}$ instead of rate $1$, i.e., in \CFFPP{} the rate of an edge depends on the vertex weights and (geometric) distances of its endpoints.

The second main difference to the proof of \Cref{thm:mainSFPlowerboundTheorem} is that cost-distances are continuous random variables, so we cannot union-bound over all possible cost-distances before and after the longest edge of a potential path anymore like we did for SFP (in \Cref{lem:mainSFPlemma}). Instead, we establish a continuous analog, a so called \emph{self-bounding inequality} that relates the expected size of a $k$-ball to itself recursively.
Another difficulty one has to overcome is that, in principle, paths of low cost-distance do not necessarily have to correspond to low graph distance as well. It could theoretically happen that many edges have very low cost and we get a low cost path which uses many edges. In such cases, we cannot use the existence of a geometrically long edge in the path, which is very central to our proof for graph distances. However, we are able to show that paths with high graph distance are actually very unlikely to have low cost-distance (see \Cref{lem:passagetimeprobbound}). Finally, a further obstacle in adapting the proof is that we have to work with the probability that a path of a certain cost exists conditioned on the weights of its endpoints at multiple points. This impacts the probabilities of edges/paths existing and thus introduces complications. To overcome this, we relate said probabilities conditional on the involved weights to their unconditional versions by employing a coupling (\Cref{prop:condweightrecpath}).

As a corollary of the above tail bound, we obtain lower bounds on the typical cost-distance similar to the one established for SFP (\Cref{cor:lowerboundSFP}). 

\begin{restatable}[Typical Graph Distances in \SFFPP{} on SFP]{corollary}{lowerbounddistancesinFPP}\label{cor:lowerboundFPP}
     Consider \SFFPP on SFP under the assumption $\alpha \in (1,2)$ and $\tau > 3$. Then, for every sufficiently small $\varepsilon > 0$ \begin{align*}
        \lim_{|x-y| \rightarrow \infty} \Pr{ \graphdistcost{x}{y} \ge \log(|x-y|)^{ \Delta(\min\{\alpha, \frac{\tau - 1}{2}\} -\varepsilon)}} = 1.
     \end{align*}
\end{restatable}

\paragraph{\textbf{Asymptotics and Probability Theory}}

We use standard Landau notation for indicating the asymptotic growth of a function. All asymptotic statements refer to the asymptotic behavior of a function as the distance $|x - y|$ tends to infinity, unless explicitly noted otherwise (like for the shape theorem, where we consider $k\to\infty$.) We further require a version of the Van den Berg-Kesten inequality (or BK-inequality) from \cite{Reimer_2000}, which gives us an upper bound on the probability of the so-called \emph{disjoint occurrence} of two events in a product probability space. For our purposes, this inequality allows us to easily handle probabilities of existence of \emph{disjoint} subpaths connecting a vertex $x$ to $u$ and a vertex $v$ to $y$, where $u, v$ are connected by an edge. In particular, we will be able to bound the probability of two such paths disjointly coexisting by the product of the probabilities of either path existing, as if they were independent (this is where the disjointness really plays a role). This disjointess is due to the fact that we are looking at shortest paths. We defer the technical details of this inequality to \Cref{sec:BK_inequality_and_proofs}.

\section{Lower Bounds for Graph Distances in SFP}\label{sec:SFP}
In this section, we provide the proof of our main lower bound. Our proof generally follows the structure of the proof of Theorem 3.1 in \cite{Biskup_2011}. Our goal is to show that the logarithmic exponent in the distances is at least roughly $\Delta(\min\{\alpha, \tauparam - 2\})$. When $\alphaparam < \tau - 2$, this is the same as $\bigDelta = \Delta(\alphaparam)$. We will first give the proof under this condition, so we first show Lemma~\ref{lem:mainSFPlemma}.

\begin{restatable}[Tail Bound for Graph Distances in SFP]{lemma}{mainSFPlemma}\label{lem:mainSFPlemma}
Consider SFP with $\alpha \in (1,2)$ such that $\alpha < \tau - 2$ and $\Delta = \Delta(\alphaparam)$. There exist constants $\ceconst, \Cconst, \pconst$ depending only on the model parameters, such that for any pair $x, y \in \mathbb{Z}^d$ and any $\hops \in \mathbb{N}$,
\begin{equation}\label{eq:WIH}
\Pr{\graphdist{x}{y} \le \hops \mid w_x, w_y} \le \IH{x}{y}{\hops}.
\end{equation}
\end{restatable}
\begin{proof}
    First of all, note that for fixed $\pconst, \Cconst$, and $\percparam$, the base case ($\hops = 1$) is true for $\ceconst$ large enough. That is because the RHS of \ref{eq:WIH} is $2^{-\pconst}e^{\ceconst}w_x^{\alphaparam}w_y^{\alphaparam}\Cconst^{-1}\geomdist{x}{y}^{-\alphaparam d}$ and the actual connection probability is at most $\percparam w_x^{\alphaparam}w_y^{\alphaparam} \geomdist{x}{y}^{-\alphaparam d}$. For the inductive step, let $h$ be our induction hypothesis, i.e., \begin{align*}
        \IHabbr{r}{\hops}{w_x}{w_y} \coloneqq w_x^{\alphaparam}w_y^{\alphaparam}\Cconst^{-1} r^{-\alphaparam d} (k + 1)^{-\pconst}e^{\ceconst k^{1/\bigDelta}}.
    \end{align*} 
    
    Assume that the induction hypothesis is true up to $\hops-1$. For $x,y$ to be connected with at most $\hops$ steps, an edge must be used with geometric distance at least $\frac{\geomdist{x}{y}}{\hops}$. This could either be the first or last edge on the path, or a so-called \emph{internal} edge. Let us first bound the probability corresponding to this edge being internal. For this, we union bound over all possible endpoints $u$ and $v$ of said edge. Actually, we integrate, since constant factors are essentially immaterial for the proof. At a given distance $r$, there are at most $c_d r^{d-1}$ vertices, for some constant $c_d$. Let $w_u, w_v$ be the weights of $u, v$, respectively. By the BK inequality (Theorem~\ref{thm:BK}), we have 
    \begin{align*}
      &\Pr{ d_G(x,y) \le k, \text{longest edge is internal } \mid w_u, w_v, w_x, w_y} \le\\
      &\hspace{.0cm}\sum_{i = 1}^{\hops - 2} \Pr{\graphdist{x}{u} \le i \mid w_x, w_u} \Pr{u \sim v \mid w_u,w_v} \Pr{\graphdist{v}{y} \le \hops - i \mid w_v, w_y} \le \lambda w_u^\alphaparam w_v^\alphaparam \left( \frac{\geomdist{x}{y}}{\hops} \right)^{-\alphaparam d} \\
      & \hspace{.5cm} \times \sum_{i = 1}^{\hops - 2} \underbrace{\left( \int_{1}^{\infty} c_d {r_u}^{d-1} \min\left\{1, \IHabbr{r_u}{i}{w_x}{w_u}\right\} \d r_u \right)}_{\I{u}{i}}  \underbrace{\left(\int_{1}^{\infty} c_d {r_v}^{d-1} \min\left\{1, \IHabbr{r_v}{k - i}{w_v}{w_y}\right\} \d r_v \right)}_{\I{v}{k - i}}.
    \end{align*}
    We will now argue that there exists a constant $\intconst{}$ (depending on the model parameters) such that
    \begin{align*}
        \I{u}{i} &\le \intconst{} \left( (i + 1)^{-\frac{\beta}{\alpha}} e^{\frac{c_1}{\alpha} i^{1/\Delta}} w_u w_x c_2^{-\frac{1}{\alpha}} \right),\\
        \I{v}{k-i} &\le \intconst{} \left( (k - i + 1)^{-\frac{\beta}{\alpha}} e^{\frac{c_1}{\alpha} (k - i )^{1/\Delta}} w_v w_y c_2^{-\frac{1}{\alpha}} \right).
    \end{align*}
    To this end, note that that there exists a value \[\hat r_u = (w_xw_u)^{\frac{1}{d}}c_2^{-\frac{1}{\alpha d}} (i+1)^{-\frac{\beta}{\alpha d}} e^{\frac{c_1}{\alpha d}(i+1)^{1/\Delta}}\] 
    for $r_u$ below which the minimum inside the integral $\I{u}{i}$ is $1$. We can thus express $\I{u}{i} = \int_1^{\hat r_u} f_1(r_u)dr_u + \int_{\hat r_u}^\infty f_2(r_u)dr_u$ where $f_1(r) = c_dr_u^{d-1}$ and $f_2$ is a polynomial in $r_u$ with exponent smaller than $-1$. Therefore, the entire integral is dominated by the value of the antiderivative of $f_1$ and $f_2$ at the splitting point $\hat r_u$. Since $f_1, f_2$ are polynomials, the antiderivative of $f_1$ is $ \le cr_uf_1(r_u)$ and the antiderivative of  $f_2$ is $\le cr_u f_2(r_u)$ for some constant $c$. Since the minimum is a continuous function, we have $f_1(\hat{r_u}) = f_2(\hat{r_u})$ and thus, $\I{u}{i} = \Theta(\hat{r} f_1(\hat{r_u})) = \Theta(\hat{r_u}^d)$ as claimed. A similar argument holds for $\I{v}{k - i}$. \footnote{This observation is helpful whenever we integrate a continuous and piecewise polynomial function.} 

    Plugging this in, we obtain \begin{align*}
        &\Pr{ d_G(x,y) \le k, \text{longest edge is internal} \mid w_u, w_v, w_x, w_y} \\
      & \hspace{0.5cm} \le \intconst{}^2 \lambda \cdot w_u^{1 + \alpha} w_v^{1 + \alpha} w_x^\alpha w_y^\alpha |x-y|^{-\alphaparam d} \Cconst^{-\frac{2}{\alphaparam}} \underbrace{k^{\alphaparam d} \sum_{i = 1}^{\hops - 2} \left((i + 1)^{-\frac{\pconst}{\alphaparam}}e^{\frac{\ceconst}{\alphaparam} {i}^{1/\bigDelta}}\right) \left((\hops - i + 1)^{-\frac{\pconst}{\alphaparam}}e^{\frac{\ceconst}{\alphaparam} (\hops - i)^{1/\bigDelta}}\right)}_{\coloneqq S}.
    \end{align*}

    Our goal now is to show that the above term is at most $\IHabbr{r}{\hops}{w_x}{w_y}$. To this end, we show that \begin{align}\label{eq:S2}
        S \le (k+1)^{-\beta} e^{\ceconst{} k^{1/\Delta}}
    \end{align} for $k$ and $\beta$ large enough. For this, notice that for small or large $i$, the exponential terms in $S$ are still quite ``tame'', due to the $\tfrac{1}{\alphaparam}$ factor. When $i$ is around $\tfrac{\hops}{2}$, their product (which is maximized for such $i$ due to concavity) is practically \begin{align*}
        \exp \left(\frac{2c_1}{\alpha} \left( \frac{k}{2} \right)^{1/\Delta} \right) = \exp \left( \frac{2c_1}{\alpha} 2^{-1/\Delta} k^{1/\Delta}  \right) = \exp \left(  c_1 k^{1/\Delta} \right)
    \end{align*}
    as $2^{-1/{\bigDelta}} = \alphaparam/2$ by definition of $\Delta$. In fact, for $i = \tfrac{\hops}{2}$, we have actual equality and for every $1 \le i \le k$, we have \begin{align*}
        \exp \left(\frac{c_1}{\alpha} i^{1/\Delta}\right) \exp \left( \frac{c_1}{\alpha} (k-i)^{1/\Delta} \right) \le \exp \left( c_1 k^{1/\Delta} \right).
    \end{align*} This is precisely the exponential term that appears in the statement we want to prove (\ref{eq:S2}). However, we also need to account for the sum and the terms polynomial in $k$ that appear in $S$. To this end, we use that -- if $i \approx k/2$ -- we gain from the product of the polynomial terms in the sum to compensate overheads. On the other hand, if $i$ is large or small, the product of the exponential terms is much smaller than what we need, so we can compensate the other terms by using the arising gap. With this in mind, we split the sum in $S$ into the cases where $|i - \tfrac{\hops}{2}| \le \tfrac{\hops}{4}$ and those where this is not true. This way, we obtain
    \begin{align}\label{eq:S}
        S \le \hops^{1 + \alpha d} e^{(1 - \gamma)\ceconst \hops^{1/\bigDelta}} + \hops^{1 + \alpha d} \left((\hops + 1)/8\right)^{-\frac{2}{\alphaparam}\pconst}e^{\ceconst \hops^{1/\bigDelta}},
    \end{align}
    where $\gamma > 0$ is a constant depending on $\bigDelta$ (and therefore on $\alphaparam$). The first term accounts for cases where $i$ is sufficiently far from $\frac{\hops}{2}$, making the exponential terms merge in a tame way. When $i \in \left[\frac{\hops}{4}, \frac{3\hops}{4}\right]$, both $i + 1$ and $k - i + 1$ are at least $\frac{\hops + 1}{8}$, since $\hops > 2$ (recall that we are analyzing the case where $\hops$ edges allow for an internal edge), and this is how the other term is obtained. 

    Now, notice that since $\alphaparam < 2$, we can choose $\pconst$ large enough such that $(\alphaparam d + 1) - \frac{2\pconst}{\alphaparam} < - \pconst$. Then, the second term in \ref{eq:S} is at most $(k+1)^{-\beta} e^{\ceconst{} k^{1/\Delta}}$ as desired. For the first term, we notice that the same holds if $k$ is large enough. Hence, for all $k$, $S$ is at most some constant $C$ times $(k+1)^{-\beta} e^{\ceconst{} k^{1/\Delta}}$. We use this to conclude that
    \begin{align*}
        &\Pr{ d_G(x,y) \le k, \text{longest edge is internal} \mid w_u, w_v, w_x, w_y} \\ 
        & \hspace{2cm} \le \funcconst \intconst^2 \lambda \Cconst^{1 - \frac{2}{\alphaparam}}  w_u^{\alphaparam + 1} w_v^{\alphaparam + 1} w_x^\alphaparam w_y^\alphaparam |x - y|^{-\alphaparam d} (\hops + 1)^{-\pconst} e^{\ceconst \hops^{1/\bigDelta}}\Cconst^{-1}\\
        & \hspace{2cm} = \funcconst \intconst^2 \lambda \Cconst^{1 - \frac{2}{\alphaparam}} w_u^{\alphaparam + 1}w_v^{\alphaparam + 1} \cdot \IHabbr{\geomdist{x}{y}}{\hops}{w_x}{w_y}.
    \end{align*}

    Since we assume that $\alphaparam < \tau - 2 $, we can 
    integrate $w_u, w_v$ out such that the corresponding integrals over $w_u$ and $w_v$ converge and only obtain another constant factor overhead. Then, we can choose $\Cconst$ large enough to compensate these constant overheads. Notice that this works since we have a factor of $\Cconst^{1 - \frac{2}{\alphaparam}}$ where the exponent is negative because $\alphaparam < 2$.
    In total, we have shown that we can choose the constants $\pconst, \ceconst$ and $\Cconst$ such that the above bound is at most $\frac{1}{3}\IHabbr{\geomdist{x}{y}}{\hops}{w_x}{w_y}$ for all $k$.

    Now, let us also bound the probability of paths in which the longest edge is adjacent to either $x$ or $y$. To this end, we sum over all possible vertices $z$ connected to $x$ by an edge of (geometric) length $\ge |x-y|/k$. Again, by the BK inequality, we have \begin{align*}
        &\Pr{d_G(x,y) \le k, \text{ longest edge incident to } x \mid w_x, w_z, w_y } \\
        &\hspace{1cm} \le \percparam w_x^\alphaparam w_z^\alphaparam \left( \frac{\geomdist{x}{y}}{\hops} \right)^{-\alphaparam d} \left( \int_{1}^{\infty} c_d {r}^{d-1} \min\left\{1, \IHabbr{r}{\hops-1}{w_z}{w_y}\right\} \d r \right)\\
        &\hspace{1cm} \le C \percparam w_x^\alphaparam w_z^\alphaparam \left( \frac{\geomdist{x}{y}}{\hops} \right)^{-\alphaparam d} \hops^{-\frac{\pconst}{\alphaparam}}e^{\frac{\ceconst}{\alphaparam} {\hops}^{1/\bigDelta}}\myweightof{z}\myweightof{y}\Cconst^{-\frac{1}{\alphaparam}}\\
        &\hspace{1cm} \le C \percparam w_z^{\alphaparam+1} \cdot (k+1)^{\beta} k^{\alpha d - \frac{\beta}{\alpha}} e^{c_1(\frac{1}{\alpha} - 1) k^{1/\Delta} } \Cconst^{1 - \frac{1}{\alpha}} \cdot \IHabbr{|x-y|}{\hops}{w_x}{w_y}.
    \end{align*}

    Again, by integrating out $w_z$, we get another constant factor. We can now choose $c_1$ large enough so that the term exponential in $\hops$ (which has a negative exponent since $\frac{1}{\alpha} - 1 < 0$) swallows the polynomial and constant terms for every $\hops$, ensuring that the factor in front of $\IHabbr{|x-y|}{\hops}{w_x}{w_y}$ is at most $\frac{1}{3}$, as desired. Finally, summing the three possibilities that the longest edge is internal, the first, or the last edge on the path yields that overall, $\Pr{d_G(x,y) \le k \mid w_x, w_y } \le \IHabbr{|x-y|}{\hops}{w_x}{w_y}$ and finishes the proof.
    \end{proof}

    As promised, we now deal with cases where $\alphaparam \ge \tauparam - 2$ by coupling SFP to SPF with larger $\alpha$ without decreasing distances using the following lemma.

    \begin{lemma}\label{lem:changealpha}
    Let $\alphaparam$ and $\percparam$ be the long-range and percolation parameters of some instance of \SFP. Fix the weights of all vertices and let $p_{uv}$ refer to the probability that two vertices $u$ and $v$ are connected by an edge. Fix some $\alphaparam' < \alphaparam$. Then, $p_{uv} \le \min\{1, \percparam^{{\alphaparam'}/{\alphaparam}}w_u^{\alphaparam'}w_v^{\alphaparam'} |u-v|^{-d{\alphaparam'}}\}$. In particular, this means that the original SFP graph (with parameter $\alphaparam$) is a subgraph of the one with parameters $\alphaparam'$ and $\percparam' = \percparam^{{\alphaparam'}/{\alphaparam}}$.
\end{lemma}
\begin{proof}
    We have 
    \begin{align*}
    p_{uv} & \le \min\{1, \percparam w_u^{\alphaparam}w_v^{\alphaparam} |u-v|^{-d{\alphaparam}}\} = \left(\min\{1, \percparam^{\frac{1}{\alphaparam}} w_uw_v |u-v|^{-d}\}\right)^{\alphaparam}\\
    & \le \left(\min\{1, \percparam^{\frac{1}{\alphaparam}} w_uw_v |u-v|^{-d}\}\right)^{\alphaparam '} = \min\{1, \percparam^{\frac{\alphaparam'}{\alphaparam}}w_u^{\alphaparam'}w_v^{\alphaparam'} |u-v|^{-d{\alphaparam'}}\}.\qedhere
    \end{align*}
\end{proof}

The implication of Lemma \ref{lem:changealpha} is that we can artificially ensure that $\alphaparam < \tauparam - 2$ by setting $\alphaparam' = \tauparam - 2 - \varepsilon$ for an arbitrarily small $\varepsilon$ and $\lambda' = \lambda^{\alphaparam'/\alphaparam}$. This allows us to prove Theorem~\ref{thm:mainSFPlowerboundTheorem} by applying \Cref{lem:mainSFPlemma} to this model since here, graph distances only get shorter due to \Cref{lem:changealpha}. We defer the proof to \Cref{sec:BK_inequality_and_proofs}, since it is only technical and the ideas in it are already presented.

\mainSFPlowerboundTheorem*

\subsection{Proof of the shape theorem}
With~\Cref{thm:mainSFPlowerboundTheorem} at hand, we can now prove~\Cref{thm:fullshapetheorem}, our shape theorem. First, we will show that w.h.p.\ (as $k$ increases) all vertices in a graph-theoretic distance of at most $k$ from a fixed vertex $x$ are found at a geometric distance of at most $r(k)$ from $x$, for some function $r(k)$.

\newcommand{\shapebadevent}{\mathcal{E}}
\newcommand{\shapeConst}{C}
\begin{thm}\label{thm:shapetheorem}
        Let $G$ be an SFP graph with parameters $\alphaparam{} \in (1, 2), \tauparam{} > 3$ and $\percparam$. 
Fix any  sufficiently small $\varepsilon > 0$ and let $\Delta' = \Delta(\min\{\alphaparam, \tauparam - 2 - \varepsilon\})$. Let $\mathcal{B}(x, k)$ denote the set of vertices reachable with at most $k$ edges from a fixed vertex $x$. Let $r(k)$ be some function of $k$ and denote by $\shapebadevent{}$ the event that there exists a vertex $y \in \mathcal{B}(x, k)$ with $|x - y| > r(k)$. We have
\[\Pr{\shapebadevent{}} \le \shapeConst{} r(k)^{d (1 - \alphaparam)} (k+1)^{-\pconst{}} \exp\left(\ceconst{} k^{1/\Delta'}\right) \]
for $\ceconst{}, \pconst{}$ as in Theorem \ref{thm:mainSFPlowerboundTheorem} and some $\shapeConst{}$ depending on the model parameters.
\end{thm}

\begin{proof}
    The proof follows by a simple union bound over all vertices at a distance of more than $r(k)$ from $x$ by invoking Theorem \ref{thm:mainSFPlowerboundTheorem} for each of them. The term $r(k)^{d (1 -\alphaparam)}$ is the result of summing the term $|x - y|^{-\alphaparam d}$ from the bound of \Cref{thm:mainSFPlowerboundTheorem} over all vertices. The constant $\shapeConst$ stems from that integration as well and hides other constants in the invocation of Theorem \ref{thm:mainSFPlowerboundTheorem}.
\end{proof}

Note that as a corollary, we get by choosing $r(k) \ge e^{\frac{\ceconst{}}{d (\alphaparam - 1)} k^{1/\Delta '}}$ that w.h.p.\ (in particular, probability at most $1 - \mathcal{O}(k^{-\pconst{}})$), the event $\shapebadevent{}$ does not occur.
\begin{corollary}\label{cor:shapecor}
    For any fixed $x\in \ZZ^d$, there is a constant $c$ such that \begin{align*}
        \lim_{k \rightarrow \infty} \Pr{\max\{|x - y| \: \mid \: y \in \mathcal{B}(x, k)\} \le \exp\left(c k^{1/\Delta'}\right) } = 1
    \end{align*}
    where $\Delta'$ is as in Theorem \ref{thm:shapetheorem} and $c$ depends on the model parameters and $\Delta'$.
    If additionally $\alphaparam < \tau - 2$, \begin{align*}
        \lim_{k \rightarrow \infty} \Pr{\max\{|x - y| \: \mid \: y \in \mathcal{B}(x, k)\} \le \exp\left(c k^{1/\Delta}\right) } = 1
    \end{align*}
    
\end{corollary}

We have thus already proven one side of the shape theorem, which we restate here for convenience and finally prove in full. Note that we have actually proven the upper bound in a slightly stronger form with $r(k) = e^{ck^{1/\Delta}}$ if $\alpha < \tau - 2$ (without the $\eps$), but we don't have the same precision for the lower bound. This lack of precision is due to the weaker assumption on the connection probability in~\cite{Biskup_2011}.

\fullshapetheorem*

\begin{proof}
We have already shown that w.h.p.\ $\mathcal{B}(x, k)\subseteq \farvertices$, which was the main part. The other relation in (\shapeEqTitle{}) follows from Theorem 1.2 in~\cite{Biskup_2011}, when $q(k) = e^{k^{1/\Delta - \epsilon}}$. In~\cite{Biskup_2011}, LRP is studied, with a less restrictive assumption on the connection probability between two vertices $x, y$. Namely, it is only assumed that the connection probability decays like $|x - y|^{-\alphaparam d + o(1)}$, which is also true in our case in particular. Moreover, one can couple SFP with LRP so that the former is a supergraph of the latter and thus has pointwise shorter graph distances. Therefore, we can invoke the results (which upper bound the distances) shown there for SFP as well. Note that when $\alphaparam < \tauparam - 2$, \Cref{cor:shapecor} can be rephrased so that $r(k) = e^{k^{1/\Delta + \varepsilon}}$, for large enough $k$. Together, these observations prove (\shapeEqTitle{}).

The second statement (\sizeEqTitle{}) follows from (\shapeEqTitle{}) by pure algebra, simply by observing that the number of vertices in $\nearvertices$ and $\farvertices$ is $\Theta(e^{dk^{1/\Delta \pm\eps}})$ respectively. Hence, (\shapeEqTitle{}) implies
\begin{align*}
    e^{dk^{1/\Delta-\eps} - \mathcal{O}(1)} \le |\mathcal{B}(x, k)| \le e^{dk^{1/\Delta+\eps} + \mathcal{O}(1)}.
\end{align*}
For sufficiently large $k$ we can swallow the factor $d$ and the term $\mathcal{O}(1)$ by slightly increasing the $\eps$ of the exponent, and obtain our desired inequality.
\end{proof}

\section{First Passage Percolation (FPP)}\label{sec:FPP}

In this section we study first passage percolation (FPP) on SFP. Recall that this means that we assign a cost to every edge which is drawn independently from an exponential distribution with rate $1$. For conciseness, we restrict ourselves to SFP even though the same technique would also work for GIRGs/HRGs. Note that we obtain the LRP model from SFP by informally setting $\tau =\infty$. Formally, since SFP is an increasing model in $\tau$ in terms of stochastic domination, the edge set of SFP with any finite $\tau$ stochastically dominates the edge set of LRP. Hence, all lower bounds on cost-distances from SFP also transfer to LRP.\footnote{The same is true for graph-distances, but here the results for LRP were already known.} In the following, we assume for simplicity that $\lambda = 1$; this does not affect our results.

In contrast to plain SFP, in FPP we have an additional source of randomness since not only the existence of an edge is random but also its cost. To prove lower bounds on cost-distances, it is therefore simpler (and sufficient) to consider a model in which there is only one source of randomness for the edges. We call this model \emph{Complete Scale Free First Passage Percolation}, or \CFFPP{} for short. Here, all edges exist a priori, i.e., the graph is fixed to be the complete graph with vertex set $\ZZ^d$. However, we now draw the cost of each edge by sampling from an exponential distribution with rate $w_u^{\alphaparam} w_v^{\alphaparam} |u - v|^{-\alphaparam d}$ (i.e.\ a rate that depends on the weights and geometric distance between the two endpoints) instead of rate $1$.

We start by showing that \SFFPP{} on SFP is dominated by \CFFPP{}, i.e., that cost-distances in \CFFPP{} can only become shorter as compared to \SFFPP{} on SFP. To that end, we need the following lemma that will allow us to combine the randomness of two events occurring with probability $\min\{1, \alphaparam\}$ and $(1 - e^{-b})$, respectively into an event occurring with probability $1 - e^{-ab}$.

\begin{lemma}\label{lem:expoprobineq}
    $\min\{1, a\}(1 - e^{-b}) \le 1 - e^{-ab}$ for all $a, b \ge 0$.
\end{lemma} \begin{proof}
    If $a \ge 1$, then the inequality is easy to see, as $(1 - e^{-b}) \le (1 - e^{-ab})$ in this case. So, let us assume that $a < 1$ from now on. Consider the function 
    \[f(b) = a(1 - e^{-b}) - (1 - e^{-ab}).\]
    Note that it suffices to show that $f(b) \le 0$ for all $b \ge 0$. We can see that $f(0) = 0$ and also 
    \[f'(b) = a(e^{-b} - e^{-ab}) \le 0.\]
    This shows that the function $f(b)$ is non-increasing and since $f(0) = 0$, we have $f(b) \le 0$ for all $b \ge 0$.
\end{proof}

With this, we establish a coupling between \SFFPP{} on SFP and \CFFPP{} such that cost-distances in \CFFPP{} are at most as large as cost-distances in \SFFPP{} on SFP.

\begin{lemma}\label{lem:sffppcoupling}
    Let $u, v$ be a pair of vertices in $\ZZ^d$. Let further $X_{(u, v)}$ be the cost of the edge $\{u, v\}$ in \emph{\SFFPP{} on SFP} if it exists, and $X_{(u, v)} = \infty$ if the edge does not exist, and let $Y_{(u, v)}$ be its cost in \emph{\CFFPP{}}. Then for any $t \ge 0$,
    \begin{align*}
        \Pr{X_{(u, v)} \le t} \le \Pr{Y_{(u, v)} \le t}.
    \end{align*}
\end{lemma}
\begin{proof}
    For the event on the LHS to be true, the edge $\{u, v\}$ must exist and then \emph{independently} the cost must be drawn to be at most $t$. The probability for the first event is $\min\{1, (w_u w_v)^{\alphaparam}|u - v|^{-\alphaparam d}\}$ and the probability of the latter is $1 - e^{-t}$. For the event on the RHS, one simply needs that the cost sampled from an exponential distribution with rate $(w_u w_v)^{\alphaparam}|u - v|^{-\alphaparam d}$ is at most $t$ and the probability of this is exactly $1 - e^{-(w_u w_v)^{\alphaparam}|u - v|^{-\alphaparam d} t}$. Lemma \ref{lem:expoprobineq} finishes the proof.
\end{proof}

Lemma \ref{lem:sffppcoupling} shows that any lower bound shown for cost-distances in \CFFPP{} will also be true for \SFFPP{} on SFP. To see more clearly why this is true, note that we can couple the models in the following way. First, we sample the weights for the vertices in exactly the same way for both models. Then, conditioned on these weights, the probability space is a product space over independent one-dimensional random variables (technically, one of them can be infinite in value, but this is not a problem for our purposes) for which the inequality in \Cref{lem:sffppcoupling} holds. With this in mind, we continue by establishing the lower bound for cost-distances in \CFFPP{}. We will generally follow similar arguments as the ones presented in \cite{completeFPP}, which studies a model similar to \CFFPP{} but without vertex weights. To establish an upper bound on the probability that the cost-distance between two vertices is at most $t$, we need a bound on the probability that the sum of exponential random variables is at most $t$, which is provided in the following lemma, which in turn is an adaptation of Lemma 2.1 in \cite{completeFPP}.

\begin{lemma}\label{lem:passagetimeprobbound}
    Let $X_1, X_2, \dots, X_k$ be i.i.d.\ exponential random variables such that the rate of $X_i$ is $(w_i w_{i + 1})^{\alphaparam}|u_i - u_{i + 1}|^{-\alphaparam d}$, for some sequence of vertices $u_j$ with corresponding weight $w_j$, with $1 \le j \le k + 1$. The $w_i$ are drawn from a power law with exponent $\tauparam$. Assume that $2 \alphaparam < \tauparam - 1$. Then, there exists a $c > 0$ depending only on $\alphaparam, \tauparam$ such that for all $t \ge 0$,
    \begin{align*}
        \Pr{\sum\limits_{i = 1}^{k} X_i \le t} \le \left(\frac{ect}{k}\right)^k \prod\limits_{i=1}^{k} |u_i - u_{i + 1}|^{-\alphaparam d},
    \end{align*}
    where the above probability is taken over the randomness of the weights and the $X_i$ values.
\end{lemma}

\begin{proof}
    Note that each $X_i = \frac{Y_i}{(w_i w_{i + 1})^{\alphaparam}|u_i - u_{i + 1}|^{-\alphaparam d}}$, where $Y_i$ is an exponential random variable with rate 1. Let us use $\rate_i = (w_i w_{i + 1})^{\alphaparam}|u_i - u_{i + 1}|^{-\alphaparam d}$ from now on. By Markov's inequality, we have
    \begin{align*}
        \Pr{\sum\limits_{i = 1}^{k} X_i \le t} = \Pr{\exp\left(-\theta \sum\limits_{i = 1}^{k} X_i\right) \ge e^{-\theta t} } \le e^{\theta t} \Expected{\exp\left(-\theta \sum\limits_{i = 1}^{k} X_i\right)}.
    \end{align*}

    Now, let us bound the expectation above. Once one fixes the weights $w_j$, each $X_i$ is independent from each other. Moreover, for $Y_i$ with rate one, it holds that $\Expected{\exp(-\theta Y_i)} = \frac{1}{1 + \theta} \le \frac{1}{\theta}$ for $\theta > 0$. So, for a fixed realization $w_1, w_2, \dots, w_{k + 1}$ of the weights, we have:
    \begin{align*}
        \Expected{\exp\left(-\theta \sum\limits_{i = 1}^{k} X_i\right) \mid w_1, w_2, \dots, w_{k + 1}} &= \prod\limits_{i = 1}^{k} \Expected{\exp\left(-\frac{\theta}{\rate_i} Y_i\right)} \\ 
        &\le \prod\limits_{i = 1}^{k} \frac{\rate_i}{\theta}\\
        &\le \theta^{-k} \prod\limits_{i = 1}^{k+1}(w_i)^{2 \alphaparam} \prod\limits_{i=1}^{k} |u_i - u_{i + 1}|^{-\alphaparam d}.
    \end{align*}

    The weight terms are raised to $2 \alphaparam$, since each weight $w_i$ enters in (at most) two $\rate_j$ as $w_i^\alphaparam$. Integrating the weights out, we see that since they are independent, one has
    \[\Expected{\exp\left(-\theta \sum\limits_{i = 1}^{k} X_i\right)} \le \left[\theta^{-k} \prod\limits_{i=1}^{k} |u_i - u_{i + 1}|^{-\alphaparam d}\right] \prod\limits_{i = 1}^{k + 1}\Expected{(w_i)^{2 \alphaparam}}.\]

    Now, since $2 \alphaparam - \tauparam < -1$, the expectations inside the rightmost product are all at most some constant $c'$. Let $c$ be e.g.\ equal to $(c')^2$ such that $c^{k} \ge (c')^{k + 1}$. Collecting the above bounds, we have
    
 \[\Pr{\sum\limits_{i = 1}^{k} X_i \le t} \le e^{\theta t} \left[\theta^{-k} \prod\limits_{i=1}^{k} |u_i - u_{i + 1}|^{-\alphaparam d}\right] c^k.\]

Setting $\theta = \frac{k}{t}$ shows the desired bound.
\end{proof}

With Lemma \ref{lem:passagetimeprobbound} at hand, we show that the expected size of the $t$-ball around the origin grows at most exponentially with $t$. We define this ball $\costballll{x}{t}$ as the set of vertices reachable from vertex $x$ with a path of cost-distance at most $t$. Exponential growth is not enough by itself for our goal of showing a polylogarithmic lower bound on the distances but is a crucial step in doing so. To do this, we modify the proof of Lemma 2.6 and Theorem 1.2 (ii) in \cite{completeFPP}. In the following, we only consider the growth of $\costballll{0}{t}$, i.e., the $t$-ball around the origin, but it is easy to see that (by symmetry) the same statements hold if we replace the origin by any vertex $x$.

\begin{thm}[Exponential Ball Growth]\label{thm:expectedballexpo}
    Let $\costballll{0}{t}$ denote the set of vertices reachable with a path of cost at most $t$ from the origin in \emph{\CFFPP{}}. If $2 \alphaparam < \tauparam - 1$, we have for some $C$ depending only on $\alphaparam$ and $\tauparam$,
    \[\Expected{|\costballll{0}{t}|} \le e^{C t}.\]
\end{thm}
\begin{proof}
We compute
    \begin{align*}
    \Expected{|\costballll{0}{t}|} & = \sum\limits_{u \in \mathbb{Z}^d} \Pr{\graphdistcost{0}{u} \le t}\\
    & \le 1 + \sum\limits_{\substack{ u \in \mathbb{Z}^d\\ u \neq 0}} \hspace{.5cm}\sum\limits_{k = 1}^{\infty}\sum_{\substack{(0, u)-\text{path $\pi$} \\ \text{of length $k$}}} \Pr{\pi \text{ has cost distance at most }t}\\
    & \overset{\text{Lemma \ref{lem:passagetimeprobbound}}}{\le} 1 + \sum\limits_{\substack{ u \in \mathbb{Z}^d\\ u \neq 0}} \sum\limits_{k = 1}^{\infty} \left(\frac{ect}{k}\right)^k \sum_{\substack{(0, u)-\text{path} \\ (0=u_1, u_2, \dots, u_{k+1} = u)}} \left[\prod\limits_{i=1}^{k} |u_i - u_{i + 1}|^{-\alphaparam d}\right].
    \end{align*}

    The constant $c$ above is as in Lemma \ref{lem:passagetimeprobbound}. The rightmost sum above can be bounded by $b^k |u|^{-\alphaparam d}$ for some $b$ depending only on $\alphaparam$. This is done by Lemma 2.5 (c) in \cite{completeFPP} (the quantity bounded there is the above sum and is defined in equation (2.3) in the page previous to that of Lemma 2.5). With that in mind, we have

    \[\Expected{|\costballll{0}{t}|} \le 1 + \left(\sum\limits_{u \in \mathbb{Z}^d, u \neq 0} |u|^{-\alphaparam d}\right) \left(\sum\limits_{k = 1}^{\infty} \left(\frac{ecbt}{k}\right)^k \right).\]

    Since $\alphaparam > 1$, the first sum above is bounded by a constant $c_1$. Moreover, note that 
    \[\sum\limits_{k = 1}^{\infty} \left(\frac{ecbt}{k}\right)^k \le \sum\limits_{k = 0}^{\infty} \frac{(ecbt)^k}{k!} - 1 = e^{ecbt} - 1.\]

    One can choose $C$ large enough so that $\Expected{|\costballll{0}{t}|} \le e^{Ct}$. To see why, note that we can freely assume $c_1 \ge 1$. Then, setting $C = ecb c_1$ suffices. That is because of the following. Let $f(x) = x^{c_1} + c_1(1 - x) - 1$. This function is decreasing from 0 to 1 and increasing afterwards. Moreover, both $f(0)$ and $f(1)$ are non-negative, hence it is non-negative for all $x \ge 0$. Setting $x = e^{ecbt}$ shows that for all $t \ge 0$,
    \[\Expected{|\costballll{0}{t}|} \le 1 + c_1(e^{ecbt} - 1) \le (e^{ecbt})^{c_1} =     e^{Ct}.\]
\end{proof}

In the following, we define 
\begin{align*}
 g(t) \coloneqq \Expected{|\costballll{0}{t}|}   
\end{align*}
and note that we have already shown that $g(t)$ grows at most exponentially. But we can do better and show that in fact it grows at most \emph{stretched} exponentially, in particular roughly as $\exp(t^{\frac{1}{\bigDelta}})$. This intuitively corresponds to the cost-distances between two vertices $u$ and $v$ growing roughly as $(\log{|u-v|})^{\bigDelta}$, and is then also used in proving the corresponding lower bound later.

To show this improved bound, we bound the crucial quantity 
\[f(\distsymbol, t) = \sup_{|u|=\distsymbol} \Pr{\graphdistcost{0}{u} \leq t} \in [0, 1],\]
that is, the highest possible probability with which a vertex connects to the origin with cost at most $t$, given that it has geometric distance $\distsymbol$. We only consider $\distsymbol, t > 0$. One can show the following bound for $f(\distsymbol, t)$.

\begin{lemma}[Towards a Self-Bounding Inequality for $g(t)$]\label{lem:fkt}
    Consider \emph{\CFFPP{}} with $2\alphaparam < \tauparam - 1$. There exist constants $c_f, \delta > 0$ depending only on $\alphaparam$ and $\tauparam$ such that
    \[f(\distsymbol, t) \le c_f\distsymbol^{-\alphaparam d}h(t),\]
    where
    \[h(t) := t^{\alpha d} \int_{0}^{t} g(t - y)(g(y) - 1)dy + e^{-\delta t}.\]
\end{lemma}

This lemma can be seen as a generalization of the technique we use to prove \Cref{lem:mainSFPlemma}: We bound the probability that two vertices at distance $r$ are connected by a path of cost at most $t$ by a term that is essentially $r^{-\alpha d}$ (which is roughly the probability that the longest edge in such a path exists) times $h(t)$ which integrates over all possible $y$ such that said edge connects the $(t-y)$-ball around $0$ and the $y$-ball around $u$. 
Using this, we can then derive a \emph{self-bounding inequality} for $g(t)$, which relates $g$ recursively to itself such that we can derive an upper bound on $g$ by solving said recursive relation using \Cref{thm:solvetheselfboundinginequalitysomehowmagicallywedontunderstandhowthisworksandrelyontheotherpaperhere} which is identical to \cite[Theorem 5.3]{completeFPP}. We derive the self-bounding inequality by summing $f(r, t)$ over all vertices and thus express $g(t)$ as a function of $h(t)$, which -- in turn -- depends on $g$. We capture this in the following lemma.

\begin{lemma}[Self-Bounding Inequality for $g(t)$]\label{lem:selfboundineq}
    Consider \CFFPP{} with $2\alpha \le \tau - 1$. There are constants $c, \delta$ such that for all $t \ge 0$ \begin{align*}
        g(t)^\alpha \le c \left(  t^{\alpha d} \int_{0}^{t} g(t - y)g(y)dy + 1\right).
    \end{align*}
\end{lemma}
\begin{proof}
To derive the self-bounding inequality for $g$ using \Cref{lem:fkt}, we upper bound $g$ by an expression involving $f(r,t)$ and then upper bound $f(r,t)$ using \Cref{lem:fkt}. Specifically, we estimate the expected size of a $t$-ball by integrating over all vertices times the respective probability $f(r,t)$. 
\begin{align*}
    \Expected{\costballll{0}{t}} = g(t) &\leq 
    1 + \int_{1}^{\infty} c_dr^{d-1} \min\{1, f(r,t)\} \d r\\
    &= 1 + \int_{1}^{(c_{f} h(t))^{\frac{1}{\alpha d}}  } c_d r^{d-1} \d r + \int_{(c_{f} h(t))^{\frac{1}{\alpha d}}  }^\infty c_dc_f r^{d-1-\alpha d} h(t) \d r
\end{align*} since for $r > (c_{f} h(t))^{\frac{1}{\alpha d}}$, the minimum is smaller than $1$ by definition of $f(r, t)$ from \Cref{lem:fkt}. Integrating out then yields, \begin{align*}
   g(t) &\leq1 +  c' \left( h(t)^{\frac{1}{\alpha}} + h(t)^{\frac{1}{\alpha} - 1}h(t) \right) \le 1 + c''h(t)^{\frac{1}{\alpha}}
\end{align*} for some constants $c', c''$ that depend on $\alphaparam, d$ and $\tauparam$. Therefore, we infer that \begin{equation}\label{eq:selfboundingineq}
(g(t) - 1)^{\alphaparam} \le c'' h(t).
\end{equation}
It can be shown that\footnote{One simply needs to consider the function $f(x) = 2^{\alphaparam - 1}(1 + (x-1)^{\alphaparam}) - x^{\alphaparam}$ restricted to $x 
    \ge 1$, which has a global minimum of 0 at $x = 2$.} $g(t)^{\alphaparam} \le 2^{\alphaparam - 1}(1 + (g(t) - 1)^{\alphaparam})$. Chaining this inequality with \eqref{eq:selfboundingineq} and replacing $h(t)$ above by its definition in \Cref{lem:fkt} we get the claimed recursive inequality for $g(t)$. In more detail, we have
    \[g(t)^{\alphaparam} \le 2^{\alphaparam - 1}(1 + (g(t) - 1)^{\alphaparam}) \le 2^{\alphaparam - 1}(1 + c''h(t))\]
    Since $h(t)$ is bounded away from zero and since $e^{-\delta t} \le 1$, it follows that there exists a $c$ such that the inequality claimed in the lemma statement holds for all $t$.

\end{proof}

It is through this inequality that a stronger bound on $g(t)$ can be derived. For this, we use Theorem 5.3 from \cite{completeFPP} directly which we restate as \Cref{thm:solvetheselfboundinginequalitysomehowmagicallywedontunderstandhowthisworksandrelyontheotherpaperhere}. It claims (among more general things) that for a given function $g(t)$, if $1 \le g(t) \le e^{Ct}$ for some constant $C$ (which we have already shown) and an inequality similar to \ref{eq:selfboundingineq} holds, then one \emph{roughly} has $g(t) \le e^{t^{1/\bigDelta}}$. Now that we have motivated Lemma \ref{lem:fkt}, let us prove it.

\begin{proof}[Proof of Lemma \ref{lem:fkt}.]
    Fix a ``target'' vertex $u$. Let us first focus on paths from $0$ to $u$ which contain at least $\rho t$ edges (we assume that this is an integer for simplicity) for some constant $\rho$ that will be determined later. If $P_{\text{long}}$ is the probability that some such path has cost-distance less than $t$, then by a simple union bound we have 
    \begin{align*}
    P_{\text{long}} & \leq \sum_{k=\rho t}^{\infty} \hspace{.5cm} \sum_{\substack{(0, u)-\text{path $\pi$}\\ \text{of length $k$}}} \Pr{\pi \text{ has cost-distance at most }t}\\
    &\overset{\text{Lemma \ref{lem:passagetimeprobbound}}}{\le} 
    e^{\theta t} \sum\limits_{k = \rho t}^{\infty} \theta^{-k} c^k \sum_{\substack{(0, u)-\text{path} \\ (u_1 = 0, u_2, \dots, u_{k+1} = u)}} \prod\limits_{i=1}^{k} |u_i - u_{i + 1}|^{-\alphaparam d}\\
    &\overset{\text{Lemma 2.5 (c) in \cite{completeFPP}}}{\le}
    e^{\theta t} |u|^{-\alphaparam d} \sum\limits_{k = \rho t}^{\infty} \left(\frac{cb}{\theta}\right)^k.
    \end{align*}

    In the second line above, we used Lemma \ref{lem:passagetimeprobbound} but the final step where $\theta$ is set to some value is not carried out. Furthermore, the constant $b$ that emerges in the third line is as in Lemma 2.5 (c) of \cite{completeFPP}. Now, setting $\theta = \rho > ecb$, we have 
    \begin{align}\label{eq:Plong}
    P_{\text{long}} \le |u|^{-\alphaparam d} e^{\rho t} 
    \frac{\left(\frac{cb}{\rho}\right)^{\rho t}}{1 - \frac{cb}{\rho}} \le |u|^{-\alphaparam d} \frac{e}{e - 1} e^{-\rho t \log{\frac{\rho}{ecb}}}.\end{align}

    Now, let us turn our attention to paths that use at most $\rho t$ edges instead and let $P_{\text{short}}$ denote the probability that such a path has cost-distance at most $t$. The idea here is to notice that a geometrically long edge $(u_1, u_2)$ must be used (similarly as in the proof of \Cref{lem:mainSFPlemma}). In particular, this edge has to cover a distance of at least $\frac{|u|}{\rho t}$, as there are at most $\rho t$ edges used to cover a distance of $|u|$. We adapt the argumentation of the proof of Lemma 5.1 in \cite{completeFPP}, which is essentially a union bound over all the possible intermediate pairs $(u_1, u_2)$. More precisely, we get that
    \begin{align}\label{eq:integrals}
        P_{\text{short}} &\le \sum_{\substack{u_1,u_2 \in \ZZ^d \\ |u_1 - u_2| \ge |u| / (\rho t) }} \Pr{ \graphdistcost{0}{u_1} + c_{(u_1, u_2)} + \graphdistcost{u_2}{u} \le t } \nonumber \\
        & \le \sum_{\substack{u_1,u_2 \in \ZZ^d \\ |u_1 - u_2| \ge |u| / (\rho t) }} \iint \Pr{ \graphdistcost{0}{u_1} + c_{(u_1, u_2)} + \graphdistcost{u_2}{u} \le t \mid w_{u_1}, w_{u_2} } \d \mu(w_{u_1}) \d \mu(w_{u_2})
    \end{align} where $\mu(w) = w^{1-\tau}$ is the probability measure of the weight distribution and 
    \begin{align}\label{eq:graphdistcostwithweights}
        &\Pr{ \graphdistcost{0}{u_1} + c_{(u_1, u_2)} + \graphdistcost{u_2}{u} \le t \mid w_{u_1}, w_{u_2} } \nonumber\\
        &\hspace{.1cm}\le \int_{0}^{t} \d \Pr{\graphdistcost{0}{u_1} \le s \mid w_{u_1}}\int_{0}^{t-s} \Pr{\graphdistcost{u_2}{u} \le y \mid w_{u_2}} w_{u_1}^\alpha w_{u_2}^\alpha |u|^{-d\alpha}(\rho t)^{\alpha d} \d y. 
    \end{align} 
    where we use a convolution over the cost of the left and right path segment and that the density of $c_{(u_1,u_2)}$ is at most $w_{u_1}^\alpha w_{u_2}^\alpha |u|^{-d\alpha} (\rho t)^{d\alpha}$. Had there not been weights involved, the proof of Lemma 5.1 in \cite{completeFPP} would show immediately that 
    \[P_{\text{short}} \le c_{\text{short}} |u|^{- \alphaparam d} h(t)\] 
    for some constant $c_{\text{short}}$ depending on $\alphaparam$. In our case we first need to get rid of the weights. To this end, we use the following simple proposition whose proof we defer for now.

    \begin{prop}\label{prop:condweightrecpath} For any $u_1 \in \ZZ^d$ and any $t > 0$, we have
        \[\Pr{\graphdistcost{0}{u_1} \le t \mid w_1} \le 2w_1^{\alphaparam}\Pr{\graphdistcost{0}{u_1} \le t}.\]
    \end{prop}
    Applying this proposition to \eqref{eq:graphdistcostwithweights}, then substituting into \eqref{eq:integrals}, taking the sum inside the integrals, and rearranging the weights, we then get \begin{align*}
        P_{\text{short}} &\le \left( \int_{1}^\infty \int_{1}^\infty 4 w_{u_1}^{2\alpha-\tau}w_{u_2}^{2\alpha-\tau} \d w_{u_1} \d w_{u_2} \right) \left( |u|^{-d\alpha}(\rho t)^{\alpha d} \int_{0}^{t} \d g(s) \int_{0}^{t-s} g(y)  \d y \right).
    \end{align*} Here, it is easy to see that the term in the second set of parentheses is at most $c_{\text{short}} |u|^{- \alphaparam d} h(t)$ for some constant $c_{\text{short}}$ as was formally shown in \cite[Lemma 5.1]{completeFPP}. The term in the first set of parentheses is some constant that only depends on $\alpha, \tau$ since the condition $2\alphaparam < \tau - 1$ ensures that the integrals converge. This constant enters into $c_{\text{short}}$.

    If we now choose $\rho > ecb$ and $\delta = \rho \log \frac{\rho}{ecb}$, then summing $P_{\text{long}}$ from \eqref{eq:Plong} and $P_{\text{short}}$ as above yields \begin{align*}
        f(\distsymbol, t) \le c_f\distsymbol^{-\alphaparam d}h(t)
    \end{align*} for some constant $c_f$ as desired.
\end{proof}

We now give the proof of the proposition deferred above.
\begin{proof}[Proof of Proposition \ref{prop:condweightrecpath}]
    Recall that the statement of the proposition is that 
    \[\Pr{\graphdistcost{0}{u_1} \le t' \mid w_1} \le 2w_1^{\alphaparam}\Pr{\graphdistcost{0}{u_1} \le t'}\] 

    We can show this by a relatively simple coupling argument. Consider a model $M$ where instead of having the vertex $u_1$ with weight $w_1$, we have a vertex $u_1'$ with weight \emph{deterministically} equal to 1. However, to decide the costs of edges adjacent to $u_1'$, we take the minimum of $\lceil w_1^{\alphaparam} \rceil$ many independent costs for each edge (each sample is exactly as it would be in \CFFPP{}, given that the weight of $u_1'$ is 1). This can also be thought of as having $\lceil w_1^{\alphaparam} \rceil$ different copies of a weight $1$ vertex and always choosing the one with minimum cost-distance when finding a path from $0$ to $u'_1$. Note that this description diverges from the one in the previous sentence if we are interested in paths between any two arbitrary vertices, but since here we are concerned only with paths from $0$ to $u'_1$, the two are equivalent. Let $e$ be an arbitrary edge adjacent to $u_1$ and $e'$ the corresponding edge of $u_1'$, where $v$ is the other vertex in both cases with weight $w_v$. One notices that
    \begin{align*}
    \Pr{c_{e'} \ge t'} &  = \left(e^{-t'w_v^{\alphaparam} |u_1' - v|^{-\alphaparam d}}\right)^{\lceil w_1^{\alphaparam} \rceil}  \le e^{-t'w_1^{\alphaparam} w_v^{\alphaparam} |u_1 - v|^{-\alphaparam d}} 
    = \Pr{c_e \ge t'}.
    \end{align*}

    We have thus far shown that 
    \[\Pr{\graphdistcost{u_1}{0} \le t' \text{ in \CFFPP{}} \mid w_1} \le \Pr{\graphdistcost{u_1'}{0} \le t' \text{ in $M$}}.\] 
    This is because for any fixed realization of the weights of other vertices and costs of edges \emph{not} adjacent to $u_1$ or $u_1'$ (all this randomness is identical between the two models), the models can be coupled so that if a path of low cost exists in \CFFPP{}, then it also exists in $M$, since the edge costs out of $u_1'$ are pointwise smaller than those of $u_1$.

    It remains to show that 
    \[\Pr{\graphdistcost{u_1'}{0}) \le t' \text{ in $M$}} \le 2w_1^\alphaparam\Pr{\graphdistcost{u_1}{0} \le t' \text{ in \CFFPP{}}}.\] 
    To see this, consider a fixed realization $\mathcal{R}$ (it suffices to show the claim under an arbitrary such $\mathcal{R}$ by the law of total probability) of the vertex weights and edge costs not adjacent to $u_1'$ (or $u_1$ equivalently), which is coupled to be the same for both \CFFPP{} and $M$. This realization defines a per-edge ``goal'' cost that is required for $u_1'$ to connect to $0$ with cost at most $t'$. In more detail, we have
    \[\graphdistcost{u_1'}{0} \le t' \iff \exists \; v \text{ such that } c_{(u_1', v)} \le  t' - \left[\graphdistcost{0}{v}\right]_{\mathcal{R}} \]\
    where the cost-distance from $0$ to $v$ is computed only using edges and costs from the realization $\mathcal{R}$.
    For example, for the edge $(0, u_1')$, this goal is $t'$ itself, since $\graphdistcost{0}{0} = 0$ under any realization. Note that by the definition of $M$, if $u_1'$ achieves some such goal cost, it is implied that at least one of the $\lceil w_1^{\alphaparam} \rceil \le 2 w_1^\alphaparam$ copies with weight 1 achieved a goal cost as well. That is,
    \[\Pr{\graphdistcost{u_1'}{0} \le t' \text{ in } M} \le 2w_1^{\alpha}\Pr{\graphdistcost{u_1}{0} \le t' \text{ in \CFFPP{}} \mid w_{u_1} = 1}\]
    \[\le 2w_1^\alphaparam\Pr{\graphdistcost{u_1}{0} \le t' \text{ in \CFFPP{}}}\]
\end{proof}

Finally, we will use Lemma \ref{lem:fkt} in conjunction with Theorem 5.3 from \cite{completeFPP} to show the desired lower bound on cost-distances. We restate that theorem here, simplified for our use case.

\begin{thm}[Theorem 5.3 from \cite{completeFPP}] \label{thm:solvetheselfboundinginequalitysomehowmagicallywedontunderstandhowthisworksandrelyontheotherpaperhere}
    Let $g(t) : [0, \infty) \to \mathbb{R}$ be a function satisfying
    \[1 \le g(t) \le e^{C t}\]
    and
    \[g(t)^{1/\theta} \leq c_h \left( 1 + t^{\beta - 1} \int_{0}^{t} g(y)g(t - y) \, dy \right)\]
    for all $t \ge 0$ for some constants $C > 0$, $\theta \in \left(\frac{1}{2}, 1\right)$, $\beta \ge 0$, and $c_h \ge 1$. Then, there exists a constant $c_{\theta} > 1$ such that $g(t) \le G(t)$ for all $t \ge 0$ where $G(t)$ is defined such that
    \[\log{G(t)} = c_{\theta}(2\lambda t)^{\log_2(2\theta)} (\log(1 + t^{\beta}))^{\log_2(1/\theta)} (1 + o(1)).\]
\end{thm}

We use the above theorem to prove our main lemma for FPP, which we restate here.

\begin{restatable}[Tail Bound for Cost-Distances in FPP{}]{lemma}{lowerboundSFFPP}\label{lem:lowerboundSFFPP}
    Consider FPP on SFP with $2\alphaparam < \tauparam - 1$ and arbitrary vertices $x, y$. There exists a constant $c$ depending only on $\alphaparam$ and $\tauparam$ such that for $\bigDelta = 1/{\log_2{(2/ \alphaparam)}}$,
    \[\log{\Pr{\graphdistcost{x}{y} \le t}} \le c(\log(1 + t))^{1-1/\bigDelta}t^{1/\bigDelta}(1 + o(1)) - \alphaparam d \log{|x - y|} + c.\]
\end{restatable}

\begin{proof}
As discussed, it suffices to show the claim for \CFFPP{}. Moreover, by translation invariance, we can replace $x$ with the origin $0$ and $y$ by $u = y - x$. By Lemma \ref{lem:selfboundineq} we have
    \begin{equation}\label{eq:Gequation}
        g(t)^{\alphaparam} \leq c_1 \left( t^{\alphaparam d} \int_{0}^{t} g(t - y)g(y) \, dy + 1 \right)
    \end{equation}
    for some $c_1, \delta > 0$ depending only on parameters of the model. 
    We also know that $g(t) \le e^{C t}$ from Theorem \ref{thm:expectedballexpo} for a $C$ with the same dependencies. Therefore, using Theorem 5.3 from \cite{completeFPP} (stated above the current theorem) with $\theta = \frac{1}{\alphaparam}, \beta = \alphaparam d + 1$ and with the inequality $1 + t^{\beta} \le (1 + t)^{\beta}$, we have $g(t) \le G(t)$, where
    \[\log{G(t)} =  c(\log(1 + t))^{1-\frac{1}{\bigDelta}}t^{\frac{1}{\bigDelta}}(1 + o(1)),\]
    where $c$ depends only on parameters of the model.
    Now, by Lemma \ref{lem:fkt}, we have
    \begin{align*}\log{\Pr{\graphdistcost{0}{u} \le t}} & \le \log{f(|u|, t)}\\
    & \leq c' - \alphaparam d \log |u| + \log \left(t^{\alphaparam d} \int_{0}^{t} g(t - y)(g(y) - 1) \, dy + e^{-\delta t} \right)
    \end{align*}
    The constant $c'$ above comes from \Cref{lem:fkt}. The final step is to notice that inequality~\eqref{eq:Gequation} is satisfied as equality for $G(t)$, as in \cite{completeFPP}. Then, substituting the expression for it in the resulting inequality gives the desired bound.
\end{proof}

    We can use the coupling employed in Theorem \ref{thm:mainSFPlowerboundTheorem} to establish the above tail bound even if $2\alphaparam \ge \tauparam - 1$ but with $\bigDelta$ replaced by $\Delta'' = \Delta(\min\{\alphaparam, \frac{\tauparam - 1}{2}\} - \eps)$ for arbitrarily small $\eps$. We then have obtained Theorem~\ref{thm:lowerboundSFPsquared}. Since the proof is almost verbatim the same as that of \Cref{thm:mainSFPlowerboundTheorem}, we omit it. Note that one minor technical step that is required additionally to get the bound claimed in \Cref{thm:lowerboundSFPsquared} using the one obtained from \Cref{lem:lowerboundSFFPP} is to introduce an auxiliary constant $\eps'$ in addition to the $\eps$ from the statement of the theorem. This swallows the $(\log(1 + t))^{1-\frac{1}{\bigDelta}}$ factor.

\bibliographystyle{abbrv}	
\bibliography{bibliography}

\appendix

\section{On the upper bound claim in \cite{hao2023graph}}\label{sec:wrong_proof}

In this section, we briefly explain the mistake in the upper bound proof for graph distances in~\cite{hao2023graph}. There, the authors add edges for some paths of length 2 in SFP, thus cutting graph distances at most in half. Let us call the resulting graph 2-SFP. Then they compare 2-SFP to LRP with different parameters. The 2-SFP graph does not dominate an LRP because the edges are correlated, but for every edge $e$, $\Pr{e \text{ open in 2-SFP}} \ge \Pr{e \text{ open in LRP}}$.

Then they make the argument that the correlations are positive, so that for every fixed $x$-$y$-path $\pi$, by the FKG inequality:
\begin{align}\label{eq:2-SFP}
\Pr{\text{$\pi$ open in 2-SFP}} \ge \Pr{\text{$\pi$ open in LRP}}.
\end{align}
That is correct, but it does not imply the statement that we would want for a suitable $k$:
\begin{align}\label{eq:2-SFP-probability}
\Pr{\text{$\exists$ $x$-$y$-path $\pi$ of length $\le k$: $\pi$ open in 2-SFP}} \ge \Pr{\text{$\exists$ $x$-$y$-path $\pi$ of length $\le k$: $\pi$ open in LRP}}. 
\end{align}
Instead,~\eqref{eq:2-SFP} only implies by summing over all $x$-$y$-paths of length at most $k$:
\begin{align}\label{eq:2-SFP-expectation}
\Expected{\text{\# of open $x$-$y$-paths of length $\le k$ in 2-SFP}} \ge \Expected{\text{\# of open $x$-$y$-paths of length $\le k$ in LRP}}.
\end{align}
However, it is not hard to see that due to the correlations the left hand side of~\eqref{eq:2-SFP-expectation} is dominated by low-probability events where the number of paths is very large. In particular, the upper bound proof for LRP is centered around the concept of \emph{hierarchies}, and the first step of a hierarchy is to find an edge of length $\Theta(|x-y|)$, where the two endpoints lie respectively close to $x$ and $y$. It is easy to see that for some parameters considered in~\cite{hao2023graph}, with high probability such edges do exist in LRP but do not exist in 2-SFP. However, the \emph{expected number} of such edges in 2-SFP (corresponding to the left hand side of~\eqref{eq:2-SFP-expectation}) is still large because the unlikely event of a vertex of weight $\Theta(|x-y|)$ at distance $|x-y|$ from $x$ induces a very large number of $2$-paths in SFP, which become edges in $2$-SFP.

Hence, the argument in~\cite{hao2023graph} shows~\eqref{eq:2-SFP-expectation} but not~\eqref{eq:2-SFP-probability}, and this is not a minor omission but a major gap. In fact, we conjecture that the upper bound statements in~\cite{hao2023graph} are false, and that the exponent $\Delta(\alpha)= 1/\log_2(2/\alpha)$ is tight throughout the polylogarithmic regime, i.e.\ for all $\tau >3$ and $\alpha \in (1,2)$.

\section{BK inequality and deferred proofs}\label{sec:BK_inequality_and_proofs}

\subsection{BK Inequality}
To formally define the BK inequality, we let $\Omega_1, \Omega_2, \ldots, \Omega_n$ be probability spaces each consisting of finitely many elements. We now consider the product probability space $\Omega = \Omega_1 \times \Omega_2 \times \ldots \times \Omega_n$ and we say that two events $A, B$ \emph{occur disjointly} -- denoted by $A \disjointoccurence B$ -- if there are disjoint sets $I, J \subseteq [n]$ that ``witness'' the occurence of $A$ and $B$, respectively. More precisely, for a set $I \subseteq [n]$ and an element $\omega \in \Omega$, we define the ``cylinder set'' $C(\omega, I)$ as the set of all $\omega' \in \Omega$ such that $\omega_i = \omega_i'$ for all $i \in I$. With this, we define the event $[A]_I = \{ \omega \in \Omega \mid C(\omega, I) \subseteq A \}$, which informally consists of all the outcomes for which it suffices to verify that $A$ occured by only looking at the coordinates in $I$. We then say that $A \disjointoccurence B$ occurs if there are disjoint $I, J \subseteq [n]$ such that $[A]_I$ and $[B]_J$ occur, i.e., we define \begin{align*}
    A \disjointoccurence B = \bigcup_{ \substack{I, J \subseteq [n] \\ I \cap J = \emptyset } } [A]_I \cap [B]_J.    
\end{align*} The BK inequality then bounds the probability of $A \disjointoccurence B$ as follows.
\begin{thm}[BK-Inequality, Theorem 1.1 in \cite{Reimer_2000}]\label{thm:BK}
    For any two events $A, B$, we have \begin{align*}
        \Pr{A \disjointoccurence B} \le \Pr{A}\Pr{B}.
    \end{align*}
\end{thm}

We will in fact need a generalization of the disjoint occurence to more than two events like the one obtained in \cite{Arratia_Garibaldi_Mower_Stark_2015}. For events $A_1, \ldots, A_k$, we define \begin{align*}
    A_1 \disjointoccurence A_2 \disjointoccurence \ldots \disjointoccurence A_k = \bigcup_{ \substack{I_1, \ldots, I_k \subseteq [n] \\ I_1, \ldots, I_k \text{ pairwise disjoint} } } [A_1]_{I_1} \cap \ldots \cap [A_k]_{I_k}.
\end{align*} One can then verify that $A \disjointoccurence B \disjointoccurence C \subseteq (A \disjointoccurence B) \disjointoccurence C$ (see \cite[Proposition 3]{Arratia_Garibaldi_Mower_Stark_2015} for a formal proof), so we can also apply the BK-inequality repeatedly. 
\begin{corollary}[Corollary 1 in \cite{Arratia_Garibaldi_Mower_Stark_2015}]\label{cor:BK}
    For events $A_1, \ldots, A_k$, we have \begin{align*}
        \Pr{A_1 \disjointoccurence A_2 \disjointoccurence \ldots \disjointoccurence A_k} \le \prod_{i=1}^k \Pr{A_i}.
    \end{align*}
\end{corollary}

\Cref{thm:BK} later allows us to deduce that the probability of finding two disjoint paths between two given pairs of vertices $(u_1, v_1)$ and $(u_2, v_2)$ is at most as large as the product of the probabilities of finding a path between $(u_1, v_1)$ and $(u_2, v_2)$, respectively. 

\subsection{Deferred proofs}

\mainSFPlowerboundTheorem*
\begin{proof}[Proof of \Cref{thm:mainSFPlowerboundTheorem}]
First of all, note that we are interested in small $\eps$ anyway, so it may be implicitly assumed that $\eps$ is small enough in the following. Let us deal with the easy case $\alphaparam < \tauparam - 2$ first. Then, the claim follows by invoking Lemma \ref{lem:mainSFPlemma} and integrating out the weights of $x, y$. We only add a constant factor in the bound of Lemma \ref{lem:mainSFPlemma} with this integration, since we \emph{always} have $\alphaparam - \tauparam < -1$ (this is a different condition than the one assumed), as $\alphaparam < 2$ and $\tauparam > 3$. This constant is absorbed into $\Cconst{}$. Moreover, we notice that in this case we even get a stronger bound, as we have $\Delta(\alphaparam)$ instead of the required $\Delta'$. 

In the case $\alphaparam \ge \tauparam - 2$, an extra step is needed to finish the proof. We use Lemma \ref{lem:changealpha} to establish a coupling between our graph of interest $G$ and another SFP graph $G'$ for which it holds that $\alphaparam' < \tauparam - 2$, where $\alphaparam' = \tau - 2 - \eps$ is the long-range parameter of $G'$. The parameter $\tauparam$ is unchanged. Note that we also have a different $\percparam'$ in $G'$, but this is immaterial as it is still a function of the original graph parameters. We can see that the connection probabilities conditioned on the weights (which are sampled in the same way for both $G$ and $G'$) are pointwise larger in $G'$. Therefore, we can establish a coupling where $G$ is always a subgraph of $G'$, meaning that it suffices to show the claim for $G'$. This is done in precisely the same way as in the first case.
\end{proof}

\paragraph{\textbf{Extending the proof of \Cref{thm:mainSFPlowerboundTheorem} to GIRG}}\label{sec:proof-GIRG} In the proof of \Cref{lem:mainSFPlemma} which is the main building block in proving \Cref{thm:mainSFPlowerboundTheorem}, we have used that the probability of connection of two vertices $u, v$ conditioned on their weights is at most 
\[\min \left\{1,  \lambda \left(\frac{w_u w_v}{|u-v|^{d}}\right)^{\alphaparam}\right\}.\]

However, we could have just as well had any constant factor in front of this term, and the proof would still work. In particular, we could have the constant $\cupp$ from equation~\eqref{eq:GIRG} in the GIRG definition.

One other point we have to address is that vertices in GIRGs have random positions, whereas for SFP we have the grid positions deterministically. However, it is still true that the expected number of GIRG vertices in some geometric ball is up to constant factors the same as the number of SFP vertices in the same ball. Therefore, even taking the randomness of the positions into account, the integrals $I_{u, i}$ etc.\ used in the proof are still valid. In particular, we can now interpret $I_{u,i}$ and $I_{v,k-i}$ as integrating (and thus taking a union bound) over the \emph{potential positions} of $u$ and $v$ respectively, and for each such potential position multiplying the other terms with the probability density function that a vertex $u$ respectively $v$ exists there (which is constant, and so can be swallowed up by $c_d$). One may worry that the probability densities for the existence of $u$ and $v$ are not independent in GIRG, but note that they are negatively correlated (a vertex at some position makes it less likely that there is a vertex at some other position), so we can upper bound the probability density for the existence of $u$ and $v$ at two specific positions by the product of the probability densities for the existence of $u$ and that of $v$ at these respective positions. The case when the longest edge is incident to either $x$ or $y$ can be handled similarly.

\end{document}